\documentclass[final,hidelinks,onefignum,onetabnum]{siamart251216}
\usepackage{preamble}
\usepackage{lipsum}
\usepackage{amsfonts}
\usepackage{graphicx}
\usepackage{epstopdf}
\ifpdf
  \DeclareGraphicsExtensions{.eps,.pdf,.png,.jpg}
\else
  \DeclareGraphicsExtensions{.eps}
\fi

\newsiamremark{remark}{Remark}
\newsiamremark{hypothesis}{Hypothesis}
\crefname{hypothesis}{Hypothesis}{Hypotheses}
\newsiamthm{claim}{Claim}

\headers{A Simple yet Highly Accurate Prediction-Correction}{T. Kamijima, N. Marumo, and A. Takeda}

\title{A Simple yet Highly Accurate Prediction-Correction Algorithm for Time-Varying Optimization
}

\author{Tomoya Kamijima\thanks{Department of Mathematical Informatics, The University of Tokyo, Tokyo, Japan (\email{kamijima-tomoya101@g.ecc.u-tokyo.ac.jp}).}
\and Naoki Marumo\thanks{Department of Mathematical Informatics, The University of Tokyo, Tokyo, Japan (\email{marumo@mist.i.u-tokyo.ac.jp}).}
\and Akiko Takeda\thanks{Department of Mathematical Informatics, The University of Tokyo, Tokyo, Japan, and Center for Advanced Intelligence Project, RIKEN, Tokyo, Japan
(\email{takeda@mist.i.u-tokyo.ac.jp}).}}

\usepackage{amsopn}

\ifpdf
\hypersetup{
  pdftitle={A Simple yet Highly Accurate Prediction-Correction Algorithm for Time-Varying Optimization},
  pdfauthor={T. Kamijima, N. Marumo, and A. Takeda}
}
\fi

\begin{document}

\maketitle

\begin{abstract}
  This paper proposes a simple yet highly accurate prediction-correction algorithm, SHARP, for unconstrained time-varying optimization problems.
  Its prediction is based on an extrapolation derived from the Lagrange interpolation of past solutions. 
  Since this extrapolation can be computed without Hessian matrices or even gradients, the computational cost is low.
  To ensure the stability of the prediction, the algorithm includes an acceptance condition that rejects the prediction when the update is excessively large.
  The proposed method achieves a tracking error of $O(h^{p})$, where $h$ is the sampling period, assuming that the $p$-th derivative of the target trajectory is bounded and the convergence of the correction step is locally linear.
  We also prove that the method can track a trajectory of stationary points even if the objective function is non-convex.
  Numerical experiments demonstrate the high accuracy of the proposed algorithm.
\end{abstract}

\begin{keywords}
time-varying optimization, prediction-correction algorithm, tracking error, non-convex optimization, Lagrange interpolation
\end{keywords}

\begin{MSCcodes}
90C26, 90C30
\end{MSCcodes}

\section{Introduction}

In this paper, we consider the unconstrained time-varying optimization problem
\begin{gather}
    \min_{\bm x\in\R^n}f(\bm x;t),\label{eq:problem}
\end{gather}
where $f(\cdot;t)\colon\R^n\to\R$ is a differentiable function that depends on a continuous-time variable $t\geq0$.
In the real world, there are many situations in which decisions must be made while the objective function is gradually changing due to environmental changes and other factors, making it necessary to solve optimization problems that change with time.
Time-varying optimization problems arise in various applications such as robotics~\cite{koppel2017d4l}, control~\cite{hours2014parametric}, signal processing~\cite{jakubiec2012d}, electronics~\cite{dall2016optimal}, and machine learning~\cite{dixit2019online}.

Time-varying optimization methods aim to track a \textit{target trajectory} $\bm x^\ast\colon\R_{\geq0}\to\R^n$ with high accuracy.
The trajectory $\bm x^\ast(t)$ is assumed to be a smooth function of $t$ and is a stationary point of $f(\cdot;t)$ for each $t$.
The problem~\cref{eq:problem} can be solved by recasting it as a sequence of time-invariant problems
\begin{gather}
    \min_{\bm x\in\R^n}f(\bm x;t_k)\quad\text{for}\quad k=0,1,\ldots,\label{eq:problem_discrete}
\end{gather}
where $t_k\coloneqq kh$ is sampling time and $h>0$ is a sampling period.
Popkov~\cite{popkov05} proposed to apply Gradient Descent (GD) to the problem~\cref{eq:problem_discrete} for each round $k$.
We call this method Time-Varying Gradient Descent (TVGD) to distinguish it from GD for time-invariant cases.
The sequence $\{\hat{\bm x}_k\}_{k}$ generated by TVGD satisfies an asymptotic error bound, $\limsup_{k\to\infty}\|\hat{\bm x}_{k}-\bm x^\ast(t_k)\|=O(h)$, assuming that $f(\cdot;t)$ is strongly convex.

To improve this bound, Simonetto et al.~\cite{simonetto16} proposed the Gradient Trajectory Tracking (GTT) algorithm, which consists of prediction and correction steps.
The prediction step computes a prediction $\hat{\bm x}_k$ of the next target point $\bm x^\ast(t_k)$ using the inverse Hessian of $f(\cdot;t_{k-1})$, and the correction step corrects the prediction to $\bm x_k$ by applying GD to \cref{eq:problem_discrete}.
GTT achieves an improved error bound, $\limsup_{k\to\infty}\|\hat{\bm x}_k-\bm x^\ast(t_k)\|=O(h^2)$, under the assumption of strong convexity.

Since GTT requires the inverse Hessian of $f$ in its prediction step, we can think of several challenges when applying it to real-world problems.
The inverse Hessian is computationally expensive, particularly for high-dimensional problems.
Even for moderate dimensions, longer computation times necessitate a larger sampling period $h$, leading to greater errors.
Furthermore, achieving higher accuracy than $O(h^2)$ requires higher-order derivatives because the prediction step is based on the Taylor expansion of $f$.
In addition, GTT cannot be directly applied to non-strongly convex problems.
In such cases, the inverse Hessian may not exist.
Even if it does exist at every round, the error bound of $O(h^2)$ is not necessarily guaranteed.

\subsection{Our Contributions}
This paper proposes a Simple yet Highly AccuRate Prediction-correction algorithm, named SHARP, for the problem~\cref{eq:problem}.
The key advantages of the proposed algorithm are as follows:
\begin{itemize}[leftmargin=*]
    \item Its prediction step is Hessian-free. 
    It relies on extrapolating past solutions and can be computed without Hessian matrices or even gradients.
    \item The algorithm can track $\bm x^\ast$ with high accuracy.
    It guarantees an asymptotic tracking error of 
    \[
        \limsup_{k\to\infty}\|\hat{\bm x}_k-\bm x^\ast(t_k)\|=O(h^{p})
    \]
    under the {local} assumptions that {the $p$-th derivative of the target trajectory is bounded} and the correction step converges locally linearly.
    {The precise statement is given in \cref{thm:error_general}.}
    \item The algorithm is applicable to non-convex functions.
    It guarantees small tracking errors for Polyak--\L ojasiewicz~(P\L) functions, defined in \cref{def:pl}, and even for general non-convex functions. 
\end{itemize}

To achieve these advantages, we leverage the Lagrange interpolation of past solutions.
The Lagrange interpolation can be computed from a linear combination of past solutions, and the coefficients are given by binomial coefficients in this setting.
In the tracking error analysis, we derive a recursion of the tracking error by leveraging the properties of the Lagrange interpolation and binomial coefficients.
This recursion-based analysis provides a clear understanding of the tracking error.
Although the analysis does not directly extend to general non-convex functions, the proposed method still guarantees an upper bound on the error in such cases.
This is due to an acceptance condition in the algorithm that rejects the prediction when the update is excessively large.

\begin{table}
    \centering
    \caption{Comparison of time-varying optimization algorithms and the asymptotic bounds on $\|\hat{\bm x}_k-\bm x^\ast(t_k)\|$. 
    SC stands for strong convexity.
    ${}^{\ast1}$ The bound is on $f(\hat{\bm x}_k;t_k)-\min_{\bm x}f(\bm x;t_k)$. 
    ${}^{\ast2}$ The bound is on $\frac1K\sum_{k=1}^K\|\nabla_{\bm x}f(\hat{\bm x}_k;t_k)\|$. 
    ${}^{*3}$ The bound holds when the Hessian is time-invariant.
    }
    \begin{tabular}{@{\hskip 2pt}c@{\hskip 2pt}|@{\hskip 4pt}c@{\hskip 0pt}c@{\hskip 0pt}c@{\hskip 0pt}c@{\hskip 0pt}c@{\hskip 0pt}}
        \toprule
        Algorithm & Hessian free & SC & Locally SC & P\L & Non-convex \\
        \midrule
        TVGD~\cite{iwakiri2024prediction,popkov05} & $\checkmark$ & $O(h)$ & & $O(h)^{*1}$ & $O(\sqrt h)^{*2}$ \\
        \cite{rostami2025time} & $\checkmark$ & $O(h)^{\ast1}$ \\
        GTT~\cite{simonetto16} & & $O(h^2)$ \\
        NTT~\cite{simonetto16} & & $O(h^2)$ \\
        \cite{jin2016continuous} & & $O(h^2)$ \\
        \cite{simonetto17} & & {$O(h^2)$} \\
        \cite{simonetto18} & & {$O(h^2)$} \\
        \cite{bastianello19} & & {$O(h^2)$} \\
        SPC~\cite{lin2019simplified} & $\checkmark$ & $O(h^2)$ \\
        \cite{zavala2010real} & & $O(h^2)$ & $O(h^2)$ \\ 
        \cite{dontchev2013euler} & & $O(h^2)$ & $O(h^2)$ \\ 
        \cite{guo2018design} & & $O(h^3)$ \\
        \cite{qiu2019new} & & $O(h^4)$ \\
        EPC~\cite{bastianello23} & $\checkmark$ & $(O(h^p)^{*3})$ \\
        \midrule
        \textbf{SHARP} & $\checkmark$ & \begin{tabular}{c}$O(h^{p})$\\(Thm.~\ref{thm:error_sc})\end{tabular} & \begin{tabular}{c}$O(h^{p})$\\(Thm.~\ref{cor:pl_local})\end{tabular} & \begin{tabular}{c}$O(h^{2})^{*1}$\\(Thm.~\ref{thm:error_pl})\end{tabular} & \begin{tabular}{c}$O(\sqrt h)^{*2}$\\(Thm.~\ref{thm:error_nc})\end{tabular} \\
        \bottomrule
    \end{tabular}
    \label{tab:error_bound}
\end{table}

\subsection{Related Work}
\cref{tab:error_bound} provides a comparison with other algorithms.
{Note that the literature adopts two different notions of tracking error.
One measures the error with respect to the last estimate made before the function $f(\cdot;t_k)$ is revealed at round $k$ (e.g., \cite{guo2018design,qiu2019new}), which corresponds to $\|\hat{\bm x}_k-\bm x^\ast(t_k)\|$.
The other measures the error after the correction step that uses the revealed function (e.g., \cite{dontchev2013euler,simonetto16}), i.e., $\|\bm x_k-\bm x^\ast(t_k)\|$.
In this work we focus on the predicted-solution error $\|\hat{\bm x}_k-\bm x^\ast(t_k)\|$, rather than the corrected-solution error $\|\bm x_k-\bm x^\ast(t_k)\|$.
A main reason is that the corrected solution $\bm x_k$ can be made arbitrarily accurate by increasing the number of correction steps, whereas the prediction $\hat{\bm x}_k$ reflects the intrinsic performance of the prediction scheme.
In addition, in the context of online learning, it is standard to evaluate the performance using the solution computed before the true function is revealed (e.g., \cite{hazan2016introduction}).
Moreover, using the convergence rate of the correction procedure, the corrected-solution error can be bounded in terms of the predicted-solution error.
For example, the prediction step of GTT and NTT~\cite{simonetto16} yields an $O(h^2)$ bound on the predicted-solution error.
Consequently, for GTT, using a constant number of gradient-descent correction iterations preserves an $O(h^2)$ bound on the corrected-solution error; for NTT, a single Newton correction yields an $O(h^4)$ bound.}

\paragraph{Prediction-correction algorithms for strongly convex functions}
Since the introduction of the prediction-correction framework by Simonetto et al.~\cite{simonetto16}, various algorithms have been developed for a wide range of strongly convex time-varying optimization. 
Simonetto and Dell'Anese~\cite{simonetto17} proposed an algorithm for convex constrained problems by using the projection operator $\proj_{\mathcal X}$ onto the feasible region $\mathcal X$.
For linearly constrained problems, Simonetto~\cite{simonetto18} proposed a dual-ascent-type algorithm that solves a convex optimization problem at each round.
Bastianello et al.~\cite{bastianello19} proposed a splitting-type algorithm for non-smooth problems.
{The asymptotic tracking error of these prediction-correction algorithms is $O(h^2)$.}

{Algorithms that exploit the temporal evolution of the optimal trajectory to improve convergence and tracking, including prediction-correction algorithms, are known as structured algorithms and have attracted increasing attention recently~\cite{simonetto2020time}.
When equipped with internal models of the time variation, structured algorithms can attain perfect tracking, i.e., the asymptotic tracking error is zero, for certain classes of temporal changes~\cite{bastianello2023internal,bianchin2024internal}.
Moreover, stochastic variants of time-varying optimization problems have also been studied~\cite{maity2022predictor,simonetto2024nonlinear}.}

\paragraph{Non-convex time-varying optimization}
Non-convex problems have also been investigated, but most studies assume local strong convexity and focus on the behavior in the neighborhood of the target trajectory. 
Zavala and Anitescu~\cite{zavala2010real} proposed an algorithm for constrained problems that solves a linear complementarity problem (LCP) at each round.
Dontchev et al.~\cite{dontchev2013euler} proposed the Euler--Newton method for parametric variational inequalities.
Non-convex problems have also been analyzed from the perspective of their continuous-time limit~\cite{ding2021time,massicot2019line,tang2022running}.
For non-convex problems without local strong convexity, Iwakiri et al.~\cite{iwakiri2024prediction} derived a tracking error bound of TVGD and proposed a prediction-correction algorithm.
{Rostami and Kia~\cite{rostami2025time} also proposed a prediction-correction algorithm which can be applied to non-convex problems, although their tracking error analysis is limited to strongly convex problems.}
Unlike previous studies, which focus on either local or global guarantees, our analysis provides both: the proposed method achieves a significantly smaller error in local regions while ensuring that the global error remains controlled even for non-convex functions.

\paragraph{Extrapolation-based prediction-correction algorithms}
{Extrapolation-based prediction-correction algorithms have attracted considerable attention because of their strong tracking performance and computational efficiency.}
{Methods that extrapolate past solutions for prediction and then correct them using a dynamical-system-based update have been studied; some achieve $O(h^3)$ or $O(h^4)$ tracking error~\cite{guo2018design,jin2016continuous,liao2015taylor,qiu2019new}.}
{Building on the extrapolation studied in \cite{zhang2019presentation}}, Qi and Zhang~\cite{qi2019new} proposed a prediction-correction scheme {that allows an arbitrary number of past points to be used for predicion;} however, their method requires the inverse Hessian, and they do not provide a tracking error analysis.
Lin et al.~\cite{lin2019simplified} proposed the Simplified Prediction-Correction algorithm (SPC){, which uses linear extrapolation of the two most recent solutions in the prediction step and avoids Hessian computations.}

The most closely related algorithm to ours is the Extrapolation-based Prediction-Correction (EPC) algorithm proposed by Bastianello et al.~\cite{bastianello23}.
The algorithm also achieves an $O(h^p)$ tracking error without using Hessian information under the assumption that the $p$-th derivative of the target trajectory is bounded.
The EPC algorithm extrapolates the objective function, while we extrapolate past iterates by using the same coefficients.
The main difference is that their analysis assumes that the Hessian is time-invariant, which means that the objective function is limited to the form of $f(\bm x;t)=f_1(\bm x;t)+f_2(\bm x)$, where $f_1$ is an affine function of $\bm x$ and $f_2$ is a time-invariant function.
In addition, their method requires gradient oracles in the past rounds, while our method does not require any gradient information in the past rounds.

\subsection{Notation}
Let $\R$ and $\R_{\geq0}$ denote the set of real and nonnegative numbers, respectively.
We denote the first, second, and $n$th derivative of a function $\bm \varphi(t)$ by $\dot{\bm \varphi}(t)$, $\ddot{\bm \varphi}(t)$, and $\bm \varphi^{(n)}(t)$, respectively. 
For the partial derivatives of $f(\bm x;t)$ with regard to $\bm x$ and $t$, we use $\nabla_{\bm x}$ and $\nabla_{t}$, respectively, and $\nabla_{\bm x\bm x}$, $\nabla_{\bm x t}$, $\nabla_{tt}$ and so on for higher-order derivatives.
We denote the binomial coefficient by $\binom{n}{k}\coloneqq n!/(k!(n-k)!)$ for $k\in\{0,\ldots,n\}$, and $\binom{n}{k}\coloneqq0$ for $k\notin\{0,\ldots,n\}$.
The norm $\|\cdot\|$ denotes the Euclidean norm for vectors and the induced operator norm for matrices and tensors.
The distance between a vector $\bm x\in\R^n$ and a set $\mathcal X\subset\R^n$ is defined by $\dist(\bm x,\mathcal X)\coloneqq\inf_{\bm y\in \mathcal X}\|\bm x-\bm y\|$.

\section{Proposed Method}\label{sec:proposed}

This section introduces {the SHARP algorithm} to solve the time-varying optimization problem~\cref{eq:problem}.
The algorithm is given in \cref{alg:hospca}, which consists of a prediction step (Lines~\ref{line:prediction1}--\ref{line:prediction5}) and a correction step (Line~\ref{line:correction1}).
For each round $k$, we compute $\hat{\bm x}_k$ as a prediction of the next target point $\bm x^\ast(t_{k})$ before $f(\cdot;t_{k})$ is revealed and correct the prediction to $\bm x_{k}$ using the revealed information of $f(\cdot;t_{k})$.

\begin{algorithm}[t]
    \caption{SHARP}
    \label{alg:hospca}
    \begin{algorithmic}[1]
        \Require $\bm x_0\in\R^n$, $P\geq1$, $0\leq v\leq\infty$
        \State $\bm x_{-P+1}=\cdots=\bm x_{-1}=\bm x_0$
        \For{$k=1,2,\ldots$}
            \For{$p=P,P-1,\ldots,1$}\label{line:prediction1}
                \Comment{Prediction}
                \State $\hat{\bm x}_k^p=\sum_{i=1}^{p}(-1)^{i-1}\binom {p}{i}\bm x_{k-i}$ \label{line:prediction2}
                \If{$\|\hat{\bm x}_k^p-\bm x_{k-1}\|\leq vh$}\label{line:prediction3}
                    \Comment{Acceptance condition}
                    \State $\hat{\bm x}_k=\hat{\bm x}_k^p$ \label{line:prediction4}
                    \State{\Break}\label{line:prediction5}
                \EndIf
            \EndFor
            \State {Approximately solve $\min_{\bm x} f(\bm x;t_{k})$ with initial point $\hat{\bm x}_k$ by e.g., \cref{eq:correction} and set $\bm x_{k}$ to be the approximate solution}\label{line:correction1}\\
            \Comment{Correction}
        \EndFor
    \end{algorithmic}
\end{algorithm}

The prediction step computes 
\begin{gather}
    \hat{\bm x}_k^p=\sum_{i=1}^{p}(-1)^{i-1}\binom {p}{i}\bm x_{k-i},\label{eq:prediction}
\end{gather}
for $p=P,P-1,\ldots,1$, where $P\geq1$ is a predetermined maximum order of prediction.
The proposed prediction scheme~\cref{eq:prediction} is reduced to the non-prediction scheme when $p=1$, and to the prediction used in the SPC algorithm~\cite{lin2019simplified} when $p=2$.

The prediction~\cref{eq:prediction} is based on the following idea.
Given the past target points $\bm x^\ast(t_{k-p}),\ldots,$ $\bm x^\ast(t_{k-1})$, 
the next target point $\bm x^\ast(t_{k})$ can be extrapolated by 
the Lagrange interpolation:
\begin{gather}
    \bm x^\ast(t_{k})\approx\sum_{i=1}^p\left(\prod_{\substack{j=1\\j\neq i}}^p\frac{t_k-t_{k-j}}{t_{k-i}-t_{k-j}}\right)\bm x^\ast(t_{k-i})
    =\sum_{i=1}^{p}(-1)^{i-1}\binom {p}{i}\bm x^\ast(t_{k-i}).\label{eq:approximation}
\end{gather}
Although $\bm x^\ast(t_{k-i})$ is unknown in reality, { we can reasonably expect $\bm x_{k-i}$ to be close to $\bm x^\ast(t_{k-i})$} since $\bm x_{k-i}$ is computed by the correction step.
Under this expectation, the prediction~\cref{eq:prediction} approximates $\bm x^\ast(t_{k})$ well; \cref{fig:prediction} presents a conceptual illustration intended to convey intuition rather than a rigorous justification.

\begin{figure}[t]
    \centering
    \begin{tikzpicture}[scale=1.33]
        \pgfmathsetmacro{\a}{1.8}
        \pgfmathsetmacro{\b}{1.8}
        \pgfmathsetmacro{\c}{1.9}
        \pgfmathsetmacro{\d}{1.8}
        \draw[thick, domain=-3:3, smooth, variable=\x, ptred] plot ({\x}, {sqrt(36-\x*\x)-4});
        \fill[red] (-2, {sqrt(32)-4}) circle (1.5pt) node [below,xshift=5pt] {$\bm x^\ast(t_{k-4})$};
        \fill[red] (-1, {sqrt(35)-4}) circle (1.5pt) node [above] {$\bm x^\ast(t_{k-3})$};
        \fill[red] (-0, {sqrt(36)-4}) circle (1.5pt) node [above] {$\bm x^\ast(t_{k-2})$};
        \fill[red] (1, {sqrt(35)-4}) circle (1.5pt) node [above] {$\bm x^\ast(t_{k-1})$};
        \fill[red] (2, {sqrt(32)-4}) circle (1.5pt) node [above] {$\bm x^\ast(t_{k})$};
        \draw[domain=-3:1,smooth,variable=\x,blue,thick,dashed] plot ({\x},{-\a/6*(\x+1)*\x*(\x-1) + \b/2*(\x+2)*\x*(\x-1) - \c/2*(\x+2)*(\x+1)*(\x-1)+\d/6*(\x+2)*(\x+1)*\x});
        \draw[domain=1:1.9,smooth,variable=\x,blue,thick,->] plot ({\x},{-\a/6*(\x+1)*\x*(\x-1) + \b/2*(\x+2)*\x*(\x-1) - \c/2*(\x+2)*(\x+1)*(\x-1)+\d/6*(\x+2)*(\x+1)*\x});
        \node[blue] at (1.1, 1.3) {Prediction};
        \fill (-2,{\a}) circle (1.5pt) node [above] {$\bm x_{k-4}$};
        \fill (-1,{\b}) circle (1.5pt) node [below] {$\bm x_{k-3}$};
        \fill (0,{\c}) circle (1.5pt) node [below] {$\bm x_{k-2}$};
        \fill (1,{\d}) circle (1.5pt) node [below] {$\bm x_{k-1}$};
        \fill[blue] (2,{-\a/6*(2+1)*2*(2-1) + \b/2*(2+2)*2*(2-1) - \c/2*(2+2)*(2+1)*(2-1)+\d/6*(2+2)*(2+1)*2}) circle (1.5pt) node [below] {$\hat{\bm x}_{k}^4$};
    \end{tikzpicture}    
    \caption{{Illustration of the prediction step.} The prediction $\hat{\bm x}_k^4$ is computed from the Lagrange interpolation of $\bm x_{k-4},\ldots,\bm x_{k-1}$, which approximates $\bm x^\ast(t_k)$ well.}
    \label{fig:prediction}
\end{figure}

An error bound for the Lagrange interpolation of scalar-valued functions can be found in standard textbooks on numerical analysis (e.g., {\cite[Theorem 6.2]{suli2003introduction}}).
The following lemma extends the error bound to the vector-valued function $\bm x^\ast$, which is useful in our analysis.

\begin{lemma}\label{lem:velocity1}
    Let $0\leq\underline k<\overline k\leq\infty$ be constants.
    Suppose that $\bm x^\ast\colon[t_{\underline k},t_{\overline k}]\to\R^n$ is $p$-times differentiable with $1\leq p\leq\overline k-\underline k$.
    Then, the following holds for all $\underline k+p\leq k\leq\overline k$:
    \begin{gather}
        \left\|\bm x^\ast(t_k)-\sum_{i=1}^{p}(-1)^{i-1}\binom {p}{i}\bm x^\ast(t_{k-i})\right\|
        \leq h^{p}\sup_{t\in[t_{\underline k},t_{\overline k}]}\big\|(\bm x^\ast)^{(p)}(t)\big\|.\label{ineq:proof_velocity2_2}
    \end{gather}
\end{lemma}

This lemma can be proved by modifying the proof for {scalar-valued} functions by replacing Rolle's theorem with the fundamental theorem of calculus.
To make this paper self-contained, we provide the proof in \cref{sec:proof_velocity1}.

Since the prediction~\cref{eq:prediction} is justified {only when $\bm x_{k-i}$ is close to $\bm x^\ast(t_{k-i})$, the algorithm needs to detect violations of this condition and reject the prediction}.
Indeed, the prediction~\cref{eq:prediction} can cause the sequence $\{\hat{\bm x}_k\}_k$ to diverge if $\bm x_{k-i}$ is far from $\bm x^\ast(t_{k-i})$.
In addition, the prediction~\cref{eq:prediction} may not work well when there is another stationary point near $\bm x^\ast(t_{k-i})$, as {illustrated} in \cref{fig:spc}.

\begin{figure}[t]
    \begin{minipage}[b]{0.49\hsize}
      \centering
      \begin{tikzpicture}
          \draw[thick, domain=-3:3, smooth, variable=\x, ptred] plot ({\x}, {sqrt(36-\x*\x)-4});
          \fill (-2.1, {sqrt(32)-3.9}) circle (2pt);
          \node at (-2.6, {sqrt(32)-3.9}) {$\bm x_{k-2}$};
          \fill (-0.1, 2.1) circle (2pt);
          \node at (-0.56, 2.25) {$\bm x_{k-1}$};
          \fill[red] (-2, {sqrt(32)-4}) circle (2pt);
          \node[red] at (-1.3, 1.5) {$\bm x^\ast(t_{k-2})$};
          \fill[red] (0, 2) circle (2pt);
          \node[red] at (0, 1.7) {$\bm x^\ast(t_{k-1})$};
          \fill[red] (2, {sqrt(32)-4}) circle (2pt);
          \node[red] at (1.3, 1.5) {$\bm x^\ast(t_{k})$};
          \fill [blue] (1.9, 2.443) circle (2pt);
          \node [blue] at (1.9, 2.743) {$\hat{\bm x}_k^2$};
          \draw[->, thick, blue] (-0.1, 2.1) -- (1.7, 2.409);
          \draw[dashed, thick, blue] (-2.1, {sqrt(32)-3.9}) -- (-0.1, 2.1);
        \end{tikzpicture}   
      \subcaption{The prediction $\hat{\bm x}_k^2$ is close to the next target point $\bm x^\ast(t_{k})$.}\label{fig:spc_1}
    \end{minipage}
    \begin{minipage}[b]{0.49\hsize}
      \centering
      \begin{tikzpicture}
        \draw[thick, domain=-3:3, smooth, variable=\x, ptred] plot ({\x}, {sqrt(36-\x*\x)-5});
        \fill (-0.1, 1.1) circle (2pt);
        \node at (-0.56, 1.15) {$\bm x_{k-1}$};
        \fill[red] (-2, {sqrt(32)-5}) circle (2pt);
        \node[red] at (-1.8, 1.1) {$\tilde{\bm x}^\ast(t_{k-2})$};
        \fill[red] (0, 1) circle (2pt);
        \node[red] at (0.3, 0.7) {$\tilde{\bm x}^\ast(t_{k-1})$};
        \fill[red] (2, {sqrt(32)-5}) circle (2pt);
        \node[red] at (1.6, 0.4) {$\tilde{\bm x}^\ast(t_{k})$};
        \fill [blue] (1.9, 2.443) circle (2pt);
        \node [blue] at (2.0, 2.1) {$\hat{\bm x}_k^2$};
        \draw[->, thick, blue] (-0.1, 1.1) -- (1.7, 2.309);
        \draw[dashed, thick, blue] (-2.1, {sqrt(32)-5.9}) -- (-0.1, 1.1);
        \draw[thick, domain=-3:3, smooth, variable=\x, ptred] plot ({\x}, {sqrt(36-\x*\x)-6});
        \fill (-2.1, {sqrt(32)-5.9}) circle (2pt);
        \node at (-2.6, {sqrt(32)-5.9}) {$\bm x_{k-2}$};
        \fill[red] (-2, {sqrt(32)-6}) circle (2pt);
        \node[red] at (-1.3, -0.5) {$\bm x^\ast(t_{k-2})$};
        \fill[red] (0, 0) circle (2pt);
        \node[red] at (0, -0.3) {$\bm x^\ast(t_{k-1})$};
        \fill[red] (2, {sqrt(32)-6}) circle (2pt);
        \node[red] at (1.3, -0.5) {$\bm x^\ast(t_{k})$};
      \end{tikzpicture}   
      \subcaption{The prediction $\hat{\bm x}_k^2$ is far from both $\bm x^\ast(t_{k})$ and $\tilde {\bm x}^\ast(t_{k})$.}\label{fig:spc_2}
    \end{minipage}
    \caption{The prediction $\hat{\bm x}_k^2=2\bm x_{k-1}-\bm x_{k-2}$ does not work well when there is another stationary point $\tilde {\bm x}^\ast(t)$ near $\bm x^\ast(t)$.}\label{fig:spc}
\end{figure}

To detect these invalid cases, we propose checking the following acceptance condition:
\begin{gather}
    \|\hat{\bm x}_k^p-\bm x_{k-1}\|\leq vh,\label{eq:acceptance}
\end{gather}
where $v\geq0$ is a predetermined threshold parameter, {which corresponds to the velocity of the target trajectory as explained in the following.}
Since the tracking error can be smaller for large $p$, we take the largest $p$ satisfying the condition~\cref{eq:acceptance}, as stated on Lines~\ref{line:prediction1}--\ref{line:prediction5} of \cref{alg:hospca}.
This condition comes from the approximation 
\[
    \bm x^\ast(t_{k})-\bm x^\ast(t_{k-1})\approx\dot{\bm x}^\ast(t_{k})h=O(h),
\]
which is valid as long as $\bm x^\ast$ has a bounded derivative.
The condition~\cref{eq:acceptance} prevents the sequence $\{\bm x_k\}_k$ from diverging, particularly when the objective function is non-strongly convex.
If the objective function is strongly convex, 
the condition becomes unnecessary (i.e., we can set $v=\infty$, as stated in \cref{thm:error_sc}).
{There is a trade-off in choosing $v$: a small $v$ improves the stability of the sequence $\{\bm x_k\}_k$, while a large $v$ allows higher-order predictions, which can reduce the tracking error.
We will discuss how to choose $v$ in \cref{sec:parameter}.}

The correction step approximately minimizes $f(\bm x;t_{k})$ to correct the prediction $\hat{\bm x}_k$ to $\bm x_{k}$.
{The method for the correction step can be chosen flexibly.
Theoretically, an $O(h^p)$ tracking error can be achieved as long as the correction step converges locally linearly, as stated in \cref{ass:correction_local}.}
For example, we can use the Gradient Descent (GD) update
\begin{gather}
    \bm x_{k}^{c+1}=\bm x_{k}^c-\alpha\nabla_{\bm x}f(\bm x_{k}^c;t_{k})\quad (c=0,\ldots,C-1)\label{eq:correction}    
\end{gather}
from the initial point $\bm x_{k}^0=\hat{\bm x}_k$, where $\alpha>0$ is a step-size parameter and $C\geq1$ is a predetermined iteration number of the correction step.

\section{{Tracking Error: From Strong Convexity to Weaker Assumptions}}\label{sec:tracking_error}

This section shows that \cref{alg:hospca} tracks a target trajectory with an $O(h^p)$ error under general assumptions.
For clarity, we first show the tracking error analysis under strong convexity in \cref{sec:sc} and then extend it to more general settings in \cref{sec:general}.
Cases where the assumptions do not hold are discussed in \cref{sec:specific2}.

\subsection{Case 1: Strongly Convex Functions}\label{sec:sc}
Throughout this subsection, the following conditions are assumed.

\begin{assumption}\label{ass:sc_case}
    Let $\mu>0$ and $L_{1,1}\geq0$ be constants.
    \begin{enuminasm}
        \item $f$ is twice continuously differentiable.\label{ass:c2}
        \item $f(\cdot;t)$ is $\mu$-strongly convex (i.e., $\nabla_{\bm x\bm x}f(\bm x;t)\succeq\mu\bm I$ for all $\bm x\in\R^n$) for all $t\geq0$.\label{ass:sc}
        \item $\|\nabla_{\bm xt}f(\bm x;t)\|\leq L_{1,1}$ for all $\bm x\in\R^n$ and $t\geq0$.\label{ass:smooth11}
    \end{enuminasm}
\end{assumption}

\cref{ass:sc_case} is standard in this field.
\cref{ass:sc} guarantees the existence and uniqueness of the optimal trajectory.
Let $\bm x^\ast(t)$ be the unique trajectory and 
\[
    \sigma_i\coloneqq\begin{dcases*}
        \sup_{t\geq0}\big\|(\bm x^\ast)^{(i)}(t)\big\|&if $\bm x^\ast$ is $i$-times differentiable for all $t\geq0$,\\
        \infty&otherwise.
    \end{dcases*}
\]
{\cref{ass:sc_case} implies that the trajectory is Lipschitz continuous as stated in the next proposition.}

\begin{proposition}\label{prop:sigma1}
    Suppose that \cref{ass:sc_case} holds.
    Then, $\sigma_1\leq L_{1,1}/\mu<\infty$ holds.
\end{proposition}

\begin{proof}
    The first-order optimality condition gives 
    \begin{gather}
        \nabla_{\bm x}f(\bm x^\ast(t);t)=\bm 0\label{eq:optimality}    
    \end{gather}
    for all $t\geq0$.
    Since $\nabla_{\bm x}f$ is continuously differentiable and $\nabla_{\bm x\bm x}f(\bm x^\ast(t);t)$ is invertible for all $t\geq0$, the implicit function theorem guarantees that $\bm x^\ast$ is differentiable for all $t\geq0$.
    Differentiating \cref{eq:optimality} by $t$ yields 
    \begin{gather}
        \nabla_{\bm x\bm x}f(\bm x^\ast(t);t)\dot{\bm x}^\ast(t)+\nabla_{\bm xt}f(\bm x^\ast(t);t)=\bm 0.\label{eq:rem_sigma}
    \end{gather}
    From $\nabla_{\bm x\bm x}f(\bm x;t)\succeq\mu\bm I$ and $\|\nabla_{\bm xt}f(\bm x;t)\|\leq L_{1,1}$, we have
    \[
        \sigma_1
        =\sup_{t\geq 0}\|\dot{\bm x}^\ast(t)\|
        =\sup_{t\geq 0}\big\|\nabla_{\bm x\bm x}f(\bm x^\ast(t);t)^{-1}\nabla_{\bm xt}f(\bm x^\ast(t);t)\big\|
        \leq\frac{L_{1,1}}{\mu}<\infty,
    \]
    which completes the proof.
\end{proof}

To guarantee tracking performance, we assume that the algorithm used in the correction step (Line~\ref{line:correction1} of \cref{alg:hospca}) satisfies the following assumption.

\begin{assumption}\label{ass:correction_sc}
    There exists a constant $\gamma\in[0,(2^P-1)^{-1})$ independent of $k$ such that $\|\bm x_{k}-\bm x^\ast(t_{k})\|\leq\gamma \|\hat{\bm x}_k-\bm x^\ast(t_k)\|$ holds for all $k\geq1$. 
\end{assumption}

We give two examples of the correction step satisfying \cref{ass:correction_sc}.
In the following examples, we assume $f(\cdot;t)$ is $L_{2,0}$-smooth (i.e., $\|\nabla_{\bm x}f(\bm x;t)-\nabla_{\bm x}f(\bm y;t)\|\leq L_{2,0}\|\bm x-\bm y\|$ for all $\bm x,\bm y\in\R^n$) for all $t\geq0$.
Let $\kappa\coloneqq L_{2,0}/\mu\geq1$ be the condition number.

\begin{example}\label{example:gd}
    {We can use GD~\cref{eq:correction} with $\alpha\in(0,2/L_{2,0})$ in the correction step.}
    The linear convergence of GD in the time-invariant setting~\cite[p.~15]{ryu2016primer} implies 
    \[
        \|\bm x_{k}-\bm x^\ast(t_{k})\|\leq\theta_1^C\|\hat{\bm x}_k-\bm x^\ast(t_k)\|
    \]
    for all $k\geq1$, where $\theta_1\coloneqq\max\{|1-\alpha\mu|,|1-\alpha L_{2,0}|\}<1$.
    When the step size is $\alpha=1/L_{2,0}$, the condition $\gamma\in[0,(2^P-1)^{-1})$ is satisfied if 
    \begin{align}
        C\geq\kappa\log(2^P-1)\label{ineq:c_lower}
    \end{align}
    holds.
    This lower bound is derived from $\log(\theta_1^{-1})=-\log\theta_1\geq1-\theta_1=\kappa^{-1}$.
    Thus, we can take $C=\lceil\kappa\log(2^P-1)\rceil$, and the computational cost per iteration becomes {$\Theta(\kappa P)$}.
\end{example}

\begin{example}\label{example:nag}
    Nesterov's accelerated gradient method~\cite[p.~94 (2.2.22)]{nesterov2018lectures} can also be used in the correction step.
    Its linear convergence in the time-invariant setting~\cite[p.~94 (2.2.23)]{nesterov2018lectures} implies
    \[
        f(\bm x_k;t_k)-f(\bm x^*;t_k)
        \leq\frac{L_{2,0}+\mu}2\exp\prn*{-\frac C{\sqrt\kappa}}\|\hat{\bm x}_k-\bm x^\ast(t_k)\|^2
    \]
    for all $k\geq1$, where $C$ is the iteration number.
    The left-hand side is lower bounded by $\frac\mu2\|\bm x_k-\bm x^\ast(t_k)\|^2$ from the strong convexity of $f(\cdot;t_k)$~\cite[Corollary 3.2.3]{nesterov2018lectures}.
    Thus, we have
    \[
        \|\bm x_{k}-\bm x^\ast(t_{k})\|\leq\sqrt{\kappa+1}\exp\prn*{-\frac C{2\sqrt\kappa}}\|\hat{\bm x}_k-\bm x^\ast(t_k)\|
    \]
    for all $k\geq1$.
    When the iteration number $C$ satisfies $C\geq\sqrt\kappa(2\log2\cdot P+\log(\kappa+1))$, the condition $\gamma\in[0,(2^P-1)^{-1})$ is satisfied.
    Thus, the computational cost per iteration becomes $\Theta(\sqrt\kappa(P+\log\kappa))$, which is faster than GD.
\end{example}

Let $e_k\coloneqq\|\hat{\bm x}_k-\bm x^\ast(t_k)\|$.
The following lemma states that the tracking error can be bounded recursively.
The first and second terms on the right-hand side of \cref{ineq:recursion_sc} represent the correction and prediction errors, respectively.

\begin{lemma}\label{lem:recursive_sc}
    Suppose that \cref{ass:sc_case,ass:correction_sc} hold.
    Set the parameter $v$ of \cref{alg:hospca} as $v=\infty$.
    Fix $p\in\{1,\ldots,P\}$ arbitrarily.
    Then, the following holds for all $k\geq P$:
    \begin{gather}
        e_{k}\leq\gamma\sum_{i=1}^{P}\binom {P}{i}e_{k-i}+2^{P-p}\sigma_{p}h^{p}.\label{ineq:recursion_sc}
    \end{gather}
\end{lemma}

\begin{proof}
    Since \cref{ineq:recursion_sc} is trivial when $\sigma_p=\infty$, we focus on the case $\sigma_p < \infty$.

    {Since $v=\infty$, the acceptance condition~\cref{eq:acceptance} always holds.}
    Rewriting the prediction~\cref{eq:prediction} yields
    \begin{align}
        \hat{\bm x}_k-\bm x^\ast(t_{k})
        &=\sum_{i=1}^{P}(-1)^{i-1}\binom{P}{i}\bm x_{k-i}-\bm x^\ast(t_{k})\\
        &=\sum_{i=1}^{P}(-1)^{i-1}\binom {P}{i}\prn[\big]{\bm x_{k-i}-\bm x^\ast(t_{k-i})}+\sum_{i=0}^{P}(-1)^{i-1}\binom {P}{i}\bm x^\ast(t_{k-i}).
    \end{align}
    By taking the norm and using the triangle inequality, we have
    \begin{gather}
        e_{k}
        \leq\sum_{i=1}^{P}\binom {P}{i}\|\bm x_{k-i}-\bm x^\ast(t_{k-i})\|+\left\|\sum_{i=0}^{P}(-1)^i \binom {P}{i}\bm x^\ast(t_{k-i})\right\|.\label{ineq:proof_sc6}
    \end{gather}
    The first term on the right-hand side can be bounded by \cref{ass:correction_sc}:
    \begin{align}
        \sum_{i=1}^{P}\binom {P}{i}\|\bm x_{k-i}-\bm x^\ast(t_{k-i})\|
        &\leq\gamma\sum_{i=1}^{P}\binom {P}{i}e_{k-i}\label{ineq:proof_sc7}
    \end{align}
    for all $k\geq P$.
    To complete the proof, we will prove
    \begin{gather}
        \left\|\sum_{i=0}^{P}(-1)^i \binom {P}{i}\bm x^\ast(t_{k-i})\right\|
        \leq2^{P-p}\sigma_{p}h^{p}\label{ineq:proof_sc9}
    \end{gather}
    for all $k\geq P$.

    Vandermonde's convolution formula $\binom{a+b}{i} = \sum_{j=0}^i \binom{a}{i-j} \binom{b}{j}$ gives
    \begin{align}
        \sum_{i=0}^{P}
        (-1)^i \binom{P}{i} \bm x^\ast(t_{k-i})
        &=
        \sum_{i=0}^{P} \sum_{j=0}^i
        (-1)^i \binom{P-p}{i-j} \binom{p}{j} \bm x^\ast(t_{k-i})\\
        &=
        \sum_{\substack{d,j\geq0\\d+j\leq P}}
        (-1)^{d+j} \binom{P-p}{d} \binom{p}{j} \bm x^\ast(t_{k-d-j})\\
        &=
        \sum_{d=0}^{P - p} \sum_{j=0}^{p}
        (-1)^{d+j} \binom{P-p}{d} \binom{p}{j} \bm x^\ast(t_{k-d-j})\\
        &=
        \sum_{d=0}^{P - p} (-1)^d
        \binom{P-p}{d}
            \sum_{j=0}^{p} (-1)^j \binom{p}{j} \bm x^\ast(t_{k-d-j})
    \end{align}
    where we set $d \coloneqq i - j$ for the second equality and omit the terms with zero binomial coefficients for the third equality.
    We evaluate the norm of the above summation using \cref{lem:velocity1} {with $(\underline k,\overline k)=(0,\infty)$}, where $k$ is replaced by $k-d$, as follows:
    \begin{align}
        \left\|\sum_{j=0}^{p} (-1)^j \binom{p}{j} \bm x^\ast(t_{(k-d)-j})\right\|
        \leq
        \sigma_p h^p.
        \label{ineq:proof_sc_marumo}
    \end{align}
    To apply the lemma, we verify that $\underline k+p \leq k-d \leq \overline k$.
    
    The upper bound is trivial, since $\overline k=\infty$.
    The lower bound is confirmed by
    \[
        k-d \geq P - (P - p) = p = \underline k + p.
    \]
    We have thus obtained \cref{ineq:proof_sc_marumo}, and hence 
    \begin{align}
        \left\|\sum_{i=0}^{P}
            (-1)^i \binom{P}{i} \bm x^\ast(t_{k-i})\right\|
        \leq
        \sum_{d=0}^{P - p} \binom{P-p}{d} \sigma_p h^p
        =
        2^{P-p}\sigma_{p}h^{p}
        \leq
        2^{P-p}\sigma_{p}h^{p},
    \end{align}
    which proves \cref{ineq:proof_sc9}.

    Combining \cref{ineq:proof_sc6,ineq:proof_sc7,ineq:proof_sc9} completes the proof.
\end{proof}

By \cref{lem:recursive_sc}, \cref{alg:hospca} guarantees an $O(h^{p})$ asymptotic tracking error.

\begin{theorem}
    [Tracking Error of SHARP, Strongly Convex]\label{thm:error_sc}
    Suppose that \cref{ass:sc_case,ass:correction_sc} hold.
    Set the parameter $v$ of \cref{alg:hospca} as $v=\infty$.
    Fix $p\in\{1,\ldots,P\}$ arbitrarily.
    Then, there exists a constant $M\geq0$ such that the following holds 
    for all $k\geq P$:
    \begin{gather}
        e_k\leq\frac{2^{P-p}\sigma_{p}}{1-(2^{P}-1)\gamma}h^{p}+M\prn*{(1+\gamma^{-1})^{1/P}-1}^{-k}.\label{ineq:error_sc0}
    \end{gather}
    Furthermore, the following holds:
    \begin{gather}
        \limsup_{k\to\infty}e_k\leq\frac{2^{P-p}\sigma_{p}}{1-(2^{P}-1)\gamma}h^{p}.\label{ineq:error_sc}
    \end{gather}
\end{theorem}

\begin{proof}
    \cref{lem:recursive_sc} guarantees that the recursion~\cref{ineq:recursion_sc} holds.
    Solving the recursion gives the theorem.
    See \cref{sec:proof_recursive0} for the details.
\end{proof}

The inequality~\cref{ineq:error_sc} guarantees an $O(h^p)$ error {for all $p\leq P$ satisfying} $\sigma_p<\infty$.
{Note that the condition $\sigma_p<\infty$ becomes more restrictive as $p$ increases, which is a price of the enhancement in the tracking error.}
The following proposition gives a sufficient condition for $\sigma_2<\infty$.
The proof is similar to that of $\sigma_1<\infty$ in \cref{prop:sigma1}.
Sufficient conditions for $\sigma_p<\infty$ for $p>2$ can be derived in the same way.

\begin{proposition}\label{prop:sigma2}
    Let $f$ be three-times continuously differentiable.
    Suppose that $f(\cdot;t)$ is $\mu$-strongly convex for all $t\geq0$ and there exist constants $L_{1,1}, L_{3,0}, L_{2,1}, L_{1,2}\geq0$ such that 
    \begin{align}
        \begin{alignedat}{2}
            \|\nabla_{\bm xt}f(\bm x;t)\|&\leq L_{1,1},\quad 
            &\|\nabla_{\bm x\bm x\bm x}f(\bm x;t)\|&\leq L_{3,0},\\
            \|\nabla_{\bm x\bm x t}f(\bm x;t)\|&\leq L_{2,1},\quad
            &\|\nabla_{\bm x tt}f(\bm x;t)\|&\leq L_{1,2}
            \label{ineq:smoothness}
        \end{alignedat}    
    \end{align}
    for all $\bm x\in\R^n$ and $t\geq0$.
    Then, $\sigma_2<\infty$ holds. 
\end{proposition}

\begin{proof}
    The equation \cref{eq:rem_sigma} holds in the same way as in the proof of \cref{prop:sigma1}.
    Since $\nabla_{\bm x}f$ is twice continuously differentiable and $\nabla_{\bm x\bm x}f(\bm x^\ast(t);t)$ is invertible for all $t\geq0$, the implicit function theorem guarantees that $\bm x^\ast$ is twice differentiable for all $t\geq0$.
    Differentiating \cref{eq:rem_sigma} by $t$ again gives
    \begin{align}
        \nabla_{\bm x\bm x\bm x}f(\bm x^\ast(t);t)[\dot{\bm x}^\ast]^2+2\nabla_{\bm x\bm xt}f(\bm x^\ast(t);t)\dot{\bm x}^\ast(t)&\\
        +\nabla_{\bm x\bm x}f(\bm x^\ast(t);t)\ddot{\bm x}^\ast(t)+\nabla_{\bm xtt}f(\bm x^\ast(t);t)&=\bm 0.
    \end{align}
    From the assumption~\cref{ineq:smoothness} and $\nabla_{\bm x\bm x}f(\bm x;t)\succeq\mu\bm I$, we have
    \[
        \sigma_2=\sup_{t\geq 0}\|\ddot{\bm x}^\ast(t)\| \leq \frac{L_{3,0}L_{1,1}^2}{\mu^3}+\frac{2L_{2,1}L_{1,1}}{\mu^2}+\frac{L_{1,2}}\mu<\infty,
    \]
    which completes the proof.
\end{proof}

\cref{prop:sigma2} combined with \cref{thm:error_sc} implies that the tracking error of the proposed method is $O(h^2)$ for strongly convex functions under the assumption~\cref{ineq:smoothness}.
These assumptions are standard to guarantee an $O(h^2)$ tracking error for strongly convex functions~\cite{lin2019simplified,simonetto16}.
The assumptions in this subsection are stated for all $t\geq0$, which can be restrictive in practice. 
Therefore, in the next subsection, we relax them to hold only on a time interval.

\begin{remark}
    The prediction scheme $\hat{\bm x}_k^2=2\bm x_{k-1}-\bm x_{k-2}$ has already been proposed by Lin et al.~\cite{lin2019simplified}.
    {However, they carry out the tracking error analysis under the implicit assumption that the computed point $\bm x_k$ can be identified with the solution of a dynamical system~\cite[eq. (7)]{lin2019simplified} derived from the continuous-time limit of the update, which is not always true in general.
    By contrast, we relax that assumption into the approximate-solution requirement of \cref{ass:correction_sc}.
    This enables us to give an explicit condition on the number $C$ of the correction step in \cref{example:gd,example:nag}. 
    Another difference is that our analysis does not rely on the assumption $\|\nabla_{\bm x}f\|\leq L_{1,0}$, which is assumed in the previous work.}
\end{remark}

\subsection{Tracking Error Analysis with Weaker Assumptions}\label{sec:general}

In the previous subsection, we analyzed the tracking error of \cref{alg:hospca} assuming the objective function is globally strongly convex for all $t\geq0$. 
However, because the asymptotic tracking error depends only on the local behavior of the objective function in $\bm x$ and $t$ around the target trajectory at the current time, this global assumption can be relaxed to appropriate local conditions.

\subsubsection{Assumptions}
Relaxed assumptions are summarized as follows.
Note that $P$ and $v$ below are input parameters of \cref{alg:hospca}.

\begin{assumption}\label{ass:trajectory_local}
    Let $0\leq\underline k<\overline k\leq\infty$ and $0<r\leq\infty$ be constants.
    There exists $\bm x^\ast\colon[t_{\underline k},t_{\overline k}]\to\R^n$ such that the following holds:
    \begin{enuminasm}
        \item $\nabla_{\bm x}f(\bm x^\ast(t);t)=\bm0$ for all $t_{\underline k}\leq t\leq t_{\overline k}$.\label{ass:trajectory_local1}
        \item $\bm x^\ast$ is differentiable on $[t_{\underline k},t_{\overline k}]$, and the following holds:
        \begin{gather}
            \sigma_1\coloneqq\sup_{t\in[t_{\underline k},t_{\overline k}]}\|\dot{\bm x}^\ast(t)\|<\begin{dcases*}
                \frac rh-v&if $r<\infty$,\\
                \infty&if $r=\infty$.\\
            \end{dcases*}\label{ineq:condition13}
        \end{gather}
        \label{ass:trajectory_local2}
        \item The constant $r$ is $\infty$, or there exists $k_0\in\{\underline k,\ldots,\overline k-P\}$ such that $\|\bm x_{k_0}-\bm x^\ast(t_{k_0})\|\leq r-(v+\sigma_1)h$, where $\bm x_{k_0}$ is the corrected solution computed on Line~\ref{line:correction1} of \cref{alg:hospca}.\label{ass:initial}
        \item If $\|\hat{\bm x}_k-\bm x^\ast(t_k)\|\leq r$ holds at round $k$ of \cref{alg:hospca}, then $\bm x_k$ computed on Line~\ref{line:correction1} satisfies $\|\bm x_{k}-\bm x^\ast(t_{k})\|\leq\gamma \|\hat{\bm x}_k-\bm x^\ast(t_k)\|$, where $\gamma$ is a constant independent of $k$ satisfying \label{ass:correction_local}
        \begin{equation}
            \gamma<\frac1{2^{P}-1}\quad\text{and}\quad\gamma\leq\begin{dcases*}
                1-\frac{v+\sigma_1}{r}h&if $r<\infty$,\\
                1&if $r=\infty$.
            \end{dcases*}\label{ineq:condition2}
        \end{equation}
    \end{enuminasm}
\end{assumption}

\cref{ass:trajectory_local1,ass:trajectory_local2} guarantee the existence of a stationary point $\bm x^\ast(t)$ {and its derivative is bounded for all $t\in[t_{\underline k},t_{\overline k}]$.}
These assumptions are essential for ensuring a small tracking error, as the trajectory $\bm x^\ast$ may otherwise become discontinuous.
The condition~\cref{ineq:condition13} is not restrictive when $h$ is sufficiently small.

\cref{ass:initial,ass:correction_local} guarantee that the correction step sufficiently reduces the error.
\cref{ass:initial} means that the sequence $\{\bm x_k\}_k$ of corrected solutions approaches sufficiently close to the trajectory $\bm x^\ast(t)$.
In existing local analyses~\cite{dontchev2013euler,zavala2010real}, it is commonly assumed that the initial point is a local optimum at the initial time, in which case \cref{ass:initial} is naturally satisfied.
\cref{ass:correction_local} states that the correction step reduces the distance to the target point by a factor $\gamma$.
The condition~\cref{ineq:condition2} states that $\gamma$ is sufficiently small, and the second inequality in \cref{ineq:condition2} is weaker than the first one when $h$ is small enough.
The constant $\gamma$ depends on the objective function and the algorithm used in the correction step.

For example, \cref{ass:sc_case} holds, $f(\cdot;t)$ is $L_{2,0}$-smooth, and the correction step employs GD with sufficiently large $C$, \cref{ass:trajectory_local} is satisfied with $(\underline k,\overline k,r)=(0,\infty,\infty)$.
Since strongly convex functions have unique minimizers, we can take $\bm x^*(t)$ satisfying \cref{ass:trajectory_local1}.
\cref{prop:sigma1} guarantees \cref{ass:trajectory_local2}.
\cref{ass:initial} is trivial when $r=\infty$.
\cref{ass:correction_local} is satisfied with $\gamma=\theta_1^C$ for sufficiently large $C$, as confirmed in \cref{example:gd}.

\cref{ass:trajectory_local} may be violated if any of the following holds: (i) $r$ is small, (ii) $\gamma$ is close to 1, (iii) $h$ is large, or (iv) $\sigma_1$ is large.
Conditions (i) and (ii) make the correction-step subproblem difficult to solve, while (iii) and (iv) cause the optimizer to change rapidly.
Thus, \cref{ass:trajectory_local} naturally characterizes features of local solutions that make accurate tracking challenging.

In the rest of this section, we restrict our discussion to a trajectory $\bm x^\ast$ satisfying \cref{ass:trajectory_local}.
For this $\bm x^\ast$, {we redefine $\sigma_i$ by restricting $t$ to $[t_{\underline k},t_{\overline k}]$ as follows:}
\begin{align}
    \sigma_i&\coloneqq\begin{dcases*}
        \sup_{t\in[t_{\underline k},t_{\overline k}]}\big\|(\bm x^\ast)^{(i)}(t)\big\|&if $\bm x^\ast$ is $i$-times differentiable on $[t_{\underline k},t_{\overline k}]$,\\
        \infty&otherwise.
    \end{dcases*}
\end{align}

\subsubsection{Acceptance of Higher-Order Predictions}

As explained in \cref{sec:proposed}, the acceptance condition \cref{eq:acceptance} rejects predictions with excessively large step sizes.
To ensure that most predictions are not rejected, this subsection provides sufficient conditions for predictions to be accepted.
In particular, we show that the prediction is always accepted after a certain round, which will be formally stated in \cref{prop:acceptance_local}.

First, we derive an upper bound on $e_{k}\coloneqq\|\hat{\bm x}_k-\bm x^\ast(t_k)\|$.

\begin{lemma}\label{lem:error_general}
    Suppose that \cref{ass:trajectory_local} holds.
    Then, $e_k\leq r$ for all $k_0+1\leq k\leq \overline k$, and 
    \[
        e_{k}\leq\frac{v+\sigma_1}{1-\gamma}h+e_{k_0}\gamma^{k-k_0}
    \]
    for all $k_0\leq k\leq\overline k$.
\end{lemma}

\begin{proof}
    We first prove 
    \begin{gather}
        e_{k}\leq \|\bm x_{k-1}-\bm x^\ast(t_{k-1})\|+(v+\sigma_1)h\label{ineq:proof_local29}
    \end{gather}
    for all $\underline k+1\leq k\leq\overline k$.
    The triangle inequality gives
    \begin{align}
        e_{k}
        &=\|\hat{\bm x}_k-\bm x^\ast(t_{k})\|\\
        &\leq\|\hat{\bm x}_k-\bm x_{k-1}\|+\|\bm x_{k-1}-\bm x^\ast(t_{k-1})\|+\|\bm x^\ast(t_{k-1})-\bm x^\ast(t_{k})\|.\label{ineq:proof_sc1}
    \end{align}
    By using the acceptance condition~\cref{eq:acceptance}, the first term of \cref{ineq:proof_sc1} is upper bounded by $vh$.
    \cref{lem:velocity1} with $p=1$ implies that the last term of \cref{ineq:proof_sc1} is upper bounded by $\sigma_1 h$.
    Thus, \cref{ineq:proof_local29} holds for all $\underline k+1\leq k\leq\overline k$.

    Next, we prove $e_k\leq r$ for all $k_0+1\leq k\leq\overline k$ by induction.
    The case $k=k_0+1$ follows from \cref{ass:initial,ineq:proof_local29}.
    Assume \cref{ineq:proof_local29} for some $k\in\{k_0+1,\ldots,\overline k-1\}$.
    Then, we have
    \[
        e_{k+1}
        \leq\|\bm x_{k}-\bm x^\ast(t_{k})\|+(v+\sigma_1)h
        \leq\gamma e_k+(v+\sigma_1)h
        \leq\gamma r+(v+\sigma_1)h
        \leq r,
    \]
    where the first inequality follows from \cref{ineq:proof_local29}; the second one, from \cref{ass:correction_local} and the induction hypothesis; the third one, from the induction hypothesis; the last one, from the second inequality of \cref{ineq:condition2}.
    Thus, $e_k\leq r$ for all $k_0+1\leq k\leq\overline k$.

    From \cref{ineq:proof_local29,ass:correction_local}, we have
    \[
        e_{k}\leq \gamma e_{k-1}+(v+\sigma_1)h
    \]
    for all $k_0+1\leq k\leq\overline k$.
    Solving this recursion gives
    \[
        e_{k}
        \leq\frac{v+\sigma_1}{1-\gamma}h+\left(e_{k_0}-\frac{v+\sigma_1}{1-\gamma}h\right)\gamma^{k-k_0}
        \leq\frac{v+\sigma_1}{1-\gamma}h+e_{k_0}\gamma^{k-k_0}
    \]
    for all $k_0\leq k\leq\overline k$.
\end{proof}

\cref{lem:error_general} gives the following sufficient condition for predictions to be accepted.

\begin{proposition}\label{prop:acceptance_local}
    Let $P\geq1$ be the maximum order of prediction in \cref{alg:hospca}.
    Suppose that \cref{ass:trajectory_local} holds.
    If
    \begin{gather}
        p=1\quad\text{or}\quad {v\geq\frac{1+(2^p-3)\gamma}{1-(2^p-1)\gamma}\sigma_1+\frac{(1-\gamma)\prn*{(2^p-2)e_{k_0}\gamma^{k-k_0-p+1}+\sigma_ph^p}}{h(1-(2^p-1)\gamma)}}\label{ineq:condition4}
    \end{gather}
    holds for some $p\in\{1,\ldots,P\}$ and $k\in\{k_0+p,\ldots,\overline k\}$, then the prediction $\hat{\bm x}_k^{p}$ satisfies the acceptance condition~\cref{eq:acceptance}.
\end{proposition}

\begin{proof}
    For the case $v=\infty$ or $p=1$, the condition~\cref{eq:acceptance} always holds.
    We consider the case $v<\infty$ and $p\geq2$ in the rest of the proof.

    We give an upper bound on $\|\hat{\bm x}_k^{p}-\bm x_{k-1}\|$ to prove \cref{eq:acceptance}.
    Rewriting the prediction~\cref{eq:prediction} yields
    \begin{align}
        \hat{\bm x}_k^{p}-\bm x_{k-1}
        &=\sum_{i=1}^{p}(-1)^{i-1}\binom{p}{i}\bm x_{k-i}-\bm x_{k-1}\\
        &=\sum_{i=2}^{p}(-1)^{i-1}\binom{p}{i}\prn[\big]{\bm x_{k-i}-\bm x^\ast(t_{k-i})}+(p-1)\prn[\big]{\bm x_{k-1}-\bm x^\ast(t_{k-1})}\\
        &\quad+\sum_{i=0}^{p}(-1)^{i-1}\binom{p}{i}\bm x^\ast(t_{k-i})+\bm x^\ast(t_{k})-\bm x^\ast(t_{k-1}).
    \end{align}
    By taking the norm and using the triangle inequality, we have
    \begin{align}
        \|\hat{\bm x}_k^{p}-\bm x_{k-1}\|\label{ineq:proof_local13}
        &\leq\sum_{i=2}^{p}\binom{p}{i}\|\bm x_{k-i}-\bm x^\ast(t_{k-i})\|+(p-1)\|\bm x_{k-1}-\bm x^\ast(t_{k-1})\|\\
        &\quad+\left\|\sum_{i=0}^{p}(-1)^{i-1}\binom{p}{i}\bm x^\ast(t_{k-i})\right\|+\|\bm x^\ast(t_{k})-\bm x^\ast(t_{k-1})\|.
    \end{align}
    We will bound each term on the right-hand side. 
    To bound the first and second terms, we use \cref{ass:correction_local,lem:error_general}:
    \begin{align}
        \|\bm x_{k-i}-\bm x^\ast(t_{k-i})\|
        \leq\gamma e_{k-i}
        &\leq\gamma\left(\frac{v+\sigma_1}{1-\gamma}h+e_{k_0}\gamma^{k-i-k_0}\right)\\
        &\leq\gamma\left(\frac{v+\sigma_1}{1-\gamma}h+e_{k_0}\gamma^{k-k_0-p}\right)\label{ineq:proof_local14}
    \end{align}
    for all $k_0+p\leq k\leq\overline k$ and $1\leq i\leq p$.
    The third term can be bounded by \cref{lem:velocity1}: 
    \begin{gather}
        \left\|\sum_{i=0}^{p}(-1)^i\binom{p}{i}\bm x^\ast(t_{k-i})\right\|
        \leq\sigma_{p}h^{p}.\label{ineq:proof_local15}
    \end{gather} 
    Similarly, the last term can be bounded by \cref{lem:velocity1} with $p=1$:
    \begin{gather}
        \|\bm x^\ast(t_{k})-\bm x^\ast(t_{k-1})\|
        \leq\sigma_1h.\label{ineq:proof_local25}
    \end{gather}
    By combining \cref{ineq:proof_local13,ineq:proof_local14,ineq:proof_local15,ineq:proof_local25}, we get
    \begin{align}
        \|\hat{\bm x}_k^{p}-\bm x_{k-1}\|
        &\leq\left(\sum_{i=2}^{p}\binom{p}{i}+(p-1)\right)\gamma\left(\frac{v+\sigma_1}{1-\gamma}h+e_{k_0}\gamma^{k-k_0-p}\right)+\sigma_1h+\sigma_{p}h^{p}\\
        &=(2^p-2)\gamma\left(\frac{v+\sigma_1}{1-\gamma}h+e_{k_0}\gamma^{k-k_0-p}\right)+\sigma_1h+\sigma_{p}h^{p}\label{ineq:proof_local26}
    \end{align}
    for all $k_0+p\leq k\leq\overline k$.
    {The second inequality in \cref{ineq:condition4} is equivalent to 
    \begin{gather}
        (2^p-2)\gamma\left(\frac{v+\sigma_1}{1-\gamma}h+e_{k_0}\gamma^{k-k_0-p}\right)+\sigma_1h+\sigma_{p}h^{p}\leq vh\label{ineq:proof_local27}
    \end{gather}
    under the condition~\cref{ineq:condition2} and $p\geq2$.}
    Combining \cref{ineq:proof_local27,ineq:proof_local26} completes the proof.
\end{proof}

This proposition states that the prediction $\hat{\bm x}_k^{p}$ is always accepted after a specific round satisfying the second inequality in \cref{ineq:condition4}, since the right-hand side is monotonically decreasing in $k$.
Although $v=\infty$ is included in \cref{ineq:condition4}, this choice leads to a poor tracking performance when \cref{ass:trajectory_local} is violated (see \cref{thm:error_pl,thm:error_nc}).

Consider the case where $(p,\sigma_p)=(6,10^6)$, $\gamma=0.01$, $h=0.01$, and $(k_0,k,e_{k_0})=(0,100,1)$
as an example.
The right-hand side of the second inequality in \cref{ineq:condition4} is 
\[
    \frac{1.61}{0.37}\sigma_1+\frac{0.99\cdot(62\cdot0.01^{95}+10^{-6})}{0.01\cdot0.63}
    \approx 4.4\sigma_1.
\]
This suggests that the lower bound on $v$ is nearly proportional to $\sigma_1$, which represents the maximum velocity of the target trajectory.

\subsubsection{Tracking Error}

In this subsection, we bound the tracking error of \cref{alg:hospca} {under \cref{ass:trajectory_local}}.
The result in this subsection {is a generalization of that in \cref{sec:sc} and} will be used to analyze the tracking error for a specific function class in the next subsection.

The following lemma corresponds to \cref{lem:recursive_sc}.

\begin{lemma}\label{lem:recursive_general2}
    Suppose that \cref{ass:trajectory_local} holds.
    Fix $p\in\{1,\ldots,P\}$ arbitrarily and let $k_1$ be the smallest $k\geq\underline k$ satisfying \cref{ineq:condition4}.
    Then, the following holds for all $k_1+P\leq k\leq\overline k$:
    \begin{gather}
        e_{k}\leq\gamma\sum_{i=1}^{P}\binom {P}{i}e_{k-i}+2^{P-p}\sigma_{p}h^{p}.\label{ineq:recursion}
    \end{gather}
\end{lemma}

\begin{proof}
    {Since the outline of the proof is the same as that of \cref{lem:recursive_sc}, we only sketch the main differences.}

    Let $p_k$ be the accepted order of prediction at round $k$, i.e., $p_k$ is the largest $q \in\{1,\ldots,P\}$ such that $\|\hat{\bm x}_k^{q}-\bm x_{k-1}\|\leq vh$.
    \cref{prop:acceptance_local} implies $p_k\geq p$ for all $k_1\leq k\leq\overline k$.
    {The remainder of the proof proceeds exactly as in the proof of \cref{lem:recursive_sc}, with $P$ replaced by $p_k$.}

    When applying \cref{lem:velocity1} to obtain \cref{ineq:proof_sc_marumo}, we need to confirm $\underline k+p\leq k-d\leq\overline k$ for all $d\leq p_k-p$.
    The upper bound is immediate since $k - d \leq k \leq \overline{k}$.
    For the lower bound, we have
    \begin{align}
        k-d \geq (k_1 + P) - (p_k - p) = k_1 + p + (P - p_k) \geq k_1 + p \geq \underline k + p,
    \end{align}
    which completes the proof.
\end{proof}

The proof of \cref{lem:recursive_general2} depends on the inequality $p_k\geq p$, which is guaranteed by \cref{prop:acceptance_local}.
Note that $p_k=p$ is not necessary.

By \cref{lem:recursive_general2}, \cref{alg:hospca} guarantees an $O(h^{p})$ asymptotic tracking error.

\begin{theorem}
    [Tracking Error of SHARP]\label{thm:error_general}
    Suppose that \cref{ass:trajectory_local} holds.
    Fix $p\in\{1,\ldots,P\}$ arbitrarily and let $k_1$ be the smallest $k\geq\underline k$ satisfying \cref{ineq:condition4}.
    Then, there exists a constant $M\geq0$ which is independent of $\overline k$ such that the following holds 
    for all $k_1+P\leq k\leq\overline k$:
    \begin{gather}
        e_k\leq\frac{2^{P-p}\sigma_{p}}{1-(2^{P}-1)\gamma}h^{p}+M\prn*{(1+\gamma^{-1})^{1/P}-1}^{-k}.\label{ineq:error_general0}
    \end{gather}
    Furthermore, if $\overline k=\infty$, the following holds:
    \begin{gather}
        \limsup_{k\to\infty}e_k\leq\frac{2^{P-p}\sigma_{p}}{1-(2^{P}-1)\gamma}h^{p}.\label{ineq:error_general}
    \end{gather}
\end{theorem}

\begin{proof}
    \cref{lem:recursive_general2} guarantees that the recursion~\cref{ineq:recursion} holds.
    Solving the recursion gives the theorem.
    See \cref{sec:proof_recursive0} for the details.
\end{proof}

As with \cref{thm:error_sc}, an $O(h^p)$ error bound holds when $\sigma_p<\infty$.
In addition, in this general setting, the error bound necessitates the parameters to be chosen carefully as discussed in \cref{sec:parameter}.

\subsection{Case 2: P\L~Functions}\label{sec:pl_local}

This subsection analyzes the tracking error of the proposed method with the GD correction for Polyak--\L ojasiewicz (P\L) functions using \cref{thm:error_general}.
Although this subsection only {deals with} 
locally P\L~functions, we can also apply \cref{thm:error_general} to other settings (e.g., locally strongly convex functions and globally P\L~functions).
Since the results are similar, we omit the details.
The correspondence of constants in \cref{alg:hospca,thm:error_general} is shown in \cref{tab:specific}.

\begin{table}[t]
    \centering
    \caption{Correspondence of constants in \cref{alg:hospca,ass:correction_local} {when the correction step is GD~\cref{eq:correction}}.
    The constant $\rho$ is deifined in \cref{cor:pl_local} and $\kappa\coloneqq L_{2,0}/\mu$.
    The column $C$ shows a lower bound when $\alpha=1/L_{2,0}$.
    The column $v$ shows whether we can take $v=\infty$ or not, and $v$ should be set to satisfy \cref{ineq:condition4} when $v<\infty$.}
    \begin{tabular}{l|ccc}
        \toprule
        Assumption & $\gamma$ & $C$ & $v$ \\
        \midrule
        $\mu$-strongly convex (\cref{thm:error_sc}) & $({1-\kappa^{-1}})^C$ & {$\Theta(\kappa P)$} & $\infty$ \\
        $\mu$-P\L & $\rho({1-\kappa^{-1}})^{C/2}$ & {$\Theta(\kappa(P+\log\kappa))$} & $\infty$ \\
        Locally $\mu$-strongly convex & $({1-\kappa^{-1}})^C$ & {$\Theta(\kappa P)$} & $<\infty$ \\
        Locally $\mu$-P\L~(\cref{cor:pl_local}) & $\rho({1-\kappa^{-1}})^{C/2}$ & {$\Theta(\kappa(P+\log\kappa))$} & $<\infty$ \\
        \bottomrule
    \end{tabular}
    \label{tab:specific}
\end{table}

{First, we introduce the P\L~condition.}
Strongly convex functions satisfy the P\L~condition (see, e.g., \cite[Theorem 2]{karimi2016linear}).

\begin{definition}
    [Polyak--\L ojasiewicz]\label{def:pl}
    Let $\mathcal X\subset\R^n$ be a nonempty set.
    For $\mu>0$, a differentiable function $f(\cdot;t)\colon\R^n\to\R$ is said to be a $\mu$-P\L~function on $\mathcal X$ if 
    \[
        \frac1{2\mu}\|\nabla_{\bm x} f(\bm x;t)\|^2\geq f(\bm x;t)-\inf_{\bm y\in\mathcal X}f(\bm y;t)
    \]
    holds for all $\bm x\in\mathcal X$.
\end{definition}

The following conditions are assumed, {in addition to \cref{ass:trajectory_local1,ass:trajectory_local2,ass:initial} and the P\L~condition}.

\begin{assumption}\label{ass:smooth}
    Let $L_{2,0}>0$ be a constant.
    \begin{enuminasm}
        \item $f(\cdot;t)$ is $L_{2,0}$-smooth 
        for all $t\geq0$.\label{ass:smooth1}
        \item Line~\ref{line:correction1} of \cref{alg:hospca} computes $\bm x_k$ by $C$ iterations of GD~\cref{eq:correction} with step size $\alpha\in(0,2/L_{2,0})$.\label{ass:smooth2}
    \end{enuminasm}
\end{assumption}

\cref{ass:smooth1} implies 
\begin{gather}
    |f(\bm y;t)-f(\bm x;t)-\mybracket{\nabla_{\bm x}f(\bm x;t),\bm y-\bm x}|\leq\frac{L_{2,0}}2\|\bm y-\bm x\|^2\label{ineq:descent_lemma}
\end{gather}
for all $\bm x,\bm y\in\R^n$ and $t\geq0$~\cite[Lemma 1.2.3]{nesterov2018lectures}.

If the P\L~condition holds on $\R^n$, GD converges linearly~\cite[Theorem 4]{polyak1963gradient}.
Under the P\L~condition, the following tracking error result is derived by using \cref{thm:error_general}.

\begin{theorem}[Tracking Error of SHARP, Locally P\L]\label{cor:pl_local}
    Suppose that \cref{ass:trajectory_local1,ass:trajectory_local2,ass:initial,ass:smooth} hold.
    Let $r$ be the constant in \cref{ass:trajectory_local}.
    Assume that there exists $\mu>0$ such that $f(\cdot;t)$ is a $\mu$-P\L~function on $\{\bm x\in\R^n\mid \|\bm x-\bm x^\ast(t)\|\leq (1+\rho)r\}$ for all $t\geq0$, where
    \[
        \rho\coloneqq\frac{1}{1-\sqrt{\theta_2}}\sqrt{\frac{\alpha L_{2,0}}{2-\alpha L_{2,0}}}>0,\quad
        \theta_2\coloneqq1-\alpha\mu(2-\alpha L_{2,0})<1,
    \]
    {and $\alpha$ is the step size in \cref{ass:smooth2}.}
    Set $C$ sufficiently large to satisfy \cref{ineq:condition2} with $\gamma\coloneqq\rho\theta_2^{C/2}$.
    Fix $p\in\{1,\ldots,P\}$ arbitrarily and let $k_1$ be the smallest $k\geq\underline k$ satisfying \cref{ineq:condition4}.
    Then, there exists a constant $M\geq0$ which is independent of $\overline k$ such that \cref{ineq:error_general0} holds with $\gamma$ defined above for all $k_1\leq k\leq\overline k$.
\end{theorem}

\begin{proof}
    First, we prove that if $e_k\leq r$, the inequality $\|\bm x_k^c-\bm x^\ast(t_k)\|\leq (1+\rho)r$ holds for all $0\leq c\leq C$ by induction.
    The case $c=0$ follows from the assumption $e_k\leq r$.
    Assume $\|\bm x_k^i-\bm x^\ast(t_k)\|\leq (1+\rho)r$ for all $i\leq c$.
    From the standard argument for GD under the P\L~condition \cite[p.~868]{polyak1963gradient}, we have
    \begin{align}
        \|\bm x_k^{i+1}-\bm x_k^i\|^2
        &\leq\frac{2\alpha}{2-\alpha L_{2,0}}\theta_2^i\prn*{f(\bm x_k^0;t_k)-f(\bm x_k^\ast;t_k)}
        \leq\frac{\alpha L_{2,0}}{2-\alpha L_{2,0}}\theta_2^ie_k^2\label{ineq:proof_pl2}
    \end{align}
    for all $i\leq c$, where the second inequality follows from \cref{ineq:descent_lemma} with $(\bm x,\bm y,t)=(\bm x_k^\ast,\bm x_k^0,t_k)$.
    The triangle inequality gives
    \begin{align}
        \|\bm x_{k}^{c+1}-\bm x^\ast(t_k)\|
        \leq e_k+\sum_{i=0}^{c}\|\bm x_{k}^{i+1}-\bm x_k^i\|
        \leq e_k+\sum_{i=0}^{c}\sqrt{\frac{\alpha L_{2,0}}{2-\alpha L_{2,0}}\theta_2^i}e_k
        &\leq (1+\rho)r,
    \end{align}
    where the second inequality follows from \cref{ineq:proof_pl2} and the last inequality from $e_k\leq r$.
    Thus, if $e_k\leq r$, the inequality $\|\bm x_k^c-\bm x^\ast(t_k)\|\leq (1+\rho)r$ holds for all $0\leq c\leq C$.

    The results established above allow us to follow an argument similar to that of \cite[Theorem 4]{polyak1963gradient}, thereby showing that $\bm x_k^C$ converges to $\bm x^\ast(t_k)$ as $C\to\infty$.
    In addition, \cref{ineq:proof_pl2} holds for all $0\leq i\leq C$.
    Using these facts and the triangle inequality yields
    \begin{align}
        \|\bm x_k-\bm x^\ast(t_k)\|
        &=\lim_{C'\to\infty}\|\bm x_k^C-\bm x_k^{C'}\|
        \leq\sum_{i=C}^\infty\|\bm x_{k}^{i+1}-\bm x_k^i\|\\
        &\leq\sum_{i=C}^{\infty}\sqrt{\frac{\alpha L_{2,0}}{2-\alpha L_{2,0}}\theta_2^i}e_k
        =\rho\theta_2^{C/2}e_k
        =\gamma e_k,
    \end{align}
    which implies that \cref{ass:correction_local} holds.

    Applying \cref{thm:error_general} completes the proof.
\end{proof}

When the step size is $\alpha=1/L_{2,0}$ and \cref{ineq:condition2} is reduced to $\gamma<(2^P-1)^{-1}$, the condition for $C$ becomes
\[
    C>\frac{2\log\frac{2^P-1}{1-\sqrt{1-\kappa^{-1}}}}{-\log(1-\kappa^{-1})},
\]
where $\kappa\coloneqq L_{2,0}/\mu\geq1$ is the condition number.
Since $-\log(1-\kappa^{-1})\geq\kappa^{-1}$ and $1-\sqrt{1-\kappa^{-1}}\geq(2\kappa)^{-1}$ hold, we get a sufficient condition: $C\geq{2\kappa(\log2\cdot(P+1)+\log\kappa)}$.
Thus, we can take $C=\lceil{2\kappa(\log2\cdot(P+1)+\log\kappa)}\rceil$, and the computational cost per iteration becomes {$\Theta(\kappa(P+\log\kappa))$}, as shown in \cref{tab:specific}.

\section{Specific Tracking Error Analysis without \cref{ass:trajectory_local}}\label{sec:specific2}

This section derives tracking error results for the case where \cref{ass:smooth} holds but {\cref{ass:trajectory_local}} does not necessarily hold. 
Note that theorems in this section also guarantee the asymptotic tracking error of TVGD because \cref{alg:hospca} with $v=0$ or $P=1$ corresponds to TVGD.

\subsection{Case 3: P\L~Functions without \cref{ass:trajectory_local}}\label{sec:pl_global}

We consider the case where the P\L~condition holds on $\R^n$ but \cref{ass:trajectory_local} does not hold.
In this setting, the definition of the tracking error needs to be modified to $\dist(\hat{\bm x}_k,\mathcal X^\ast(t_k))$, where $\mathcal X^\ast(t)\coloneqq\argmin_{\bm x\in\R^n}f(\bm x;t)$, and we can no longer obtain an $O(h^{p})$ tracking error.
We also give an asymptotic error bound on $f(\hat{\bm x}_k;t_k)-f^\ast(t_k)$ as a corollary, where $f^\ast(t)\coloneqq\min_{\bm x\in\R^n}f(\bm x;t)$.

The main assumptions in this subsection are summarized as follows:

\begin{assumption}\label{ass:pl_global}
    Let $\mu>0$ and $L_{1,1}>0$ be constants.
    \begin{enuminasm}
        \item The function $f(\cdot;t)$ is a $\mu$-P\L~function on $\R^n$ for all $t\geq0$.\label{ass:pl_global1}
        \item $\mathcal X^\ast(t)\neq\emptyset$ for all $t\geq0$.\label{ass:pl_global2}
        \item $\|\nabla_{\bm x}f(\bm x;s)-\nabla_{\bm x}f(\bm x;t)\|\leq L_{1,1}|s-t|$ for all $\bm x\in\R^n$ and $s,t\geq0$.\label{ass:pl_global3}
    \end{enuminasm}
\end{assumption}

This assumption guarantees that the set of optima is Lipschitz continuous in the following sense. 

\begin{lemma}\label{lem:lipschitz}
    Suppose that \cref{ass:smooth} and \cref{ass:pl_global} hold.
    Then the following holds for all $s,t\geq0$:
    \begin{gather}
        \disth\prn*{\mathcal X^\ast(s),\mathcal X^\ast(t)}\leq \frac{L_{1,1}}\mu|s-t|,\label{ineq:lipschitz}
    \end{gather}
    where $\disth(\cdot,\cdot)$ is the Hausdorff distance, defined by 
    \[
        \disth(\mathcal X,\mathcal Y)\coloneqq\max\left\{\sup_{\bm x\in\mathcal X}\dist(\bm x,\mathcal Y),\ \sup_{\bm y\in\mathcal Y}\dist(\bm y,\mathcal X)\right\}.
    \]
\end{lemma}

\begin{proof}
    Note that $\mu$-P\L~functions satisfy the error bound condition $\|\nabla_{\bm x} f(\bm x;t)\|\geq\mu\dist(\bm x,\mathcal X^\ast(t))$~\cite[Theorem 2]{karimi2016linear}.
    For any $s,t\geq0$, we have
    \begin{align}
        \sup_{\bm x\in \mathcal X^\ast(s)}\dist\prn*{\bm x,\mathcal X^\ast(t)}
        &\leq\sup_{\bm x\in \mathcal X^\ast(s)}\frac1\mu\|\nabla_{\bm x}f(\bm x;t)\|\\
        &\leq\sup_{\bm x\in \mathcal X^\ast(s)}\frac1\mu\prn[\big]{\|\nabla_{\bm x}f(\bm x;t)-\nabla_{\bm x}f(\bm x;s)\|+\|\nabla_{\bm x}f(\bm x;s)\|}\\
        &\leq\frac{L_{1,1}}\mu|s-t|,
    \end{align}
    where the first inequality follows from the error bound condition and the last from \cref{ass:pl_global3} and $\nabla_{\bm x}f(\bm x;s)=\bm 0$ for all $\bm x\in \mathcal X^\ast(s)$.
    In the same way, $\sup_{\bm y\in \mathcal X^\ast(t)}\dist(\bm y,\mathcal X^\ast(s))\leq(L_{1,1}/\mu)|s-t|$ follows.
    Therefore, we obtain \cref{ineq:lipschitz}.
\end{proof}

\cref{lem:lipschitz} yields the following recursive relationship for the tracking error.

\begin{lemma}\label{lem:recursive2}
    Suppose that \cref{ass:smooth} and \cref{ass:pl_global} hold.
    Then, the following holds for all $k\geq2$:
    \begin{gather}
        \dist\prn*{\hat{\bm x}_k,\mathcal X^\ast(t_{k})}\leq\sqrt{\kappa\theta_2^{C}}\dist\prn*{\hat{\bm x}_{k-1},\mathcal X^\ast(t_{k-1})}+\left(v+\frac{L_{1,1}}\mu\right) h,\label{ineq:recursive2}
    \end{gather}
    where $\kappa\coloneqq L_{2,0}/\mu$ and $\theta_2\coloneqq1-\alpha\mu(2-\alpha L_{2,0})<1$.
\end{lemma}

\begin{proof}
    The triangle inequality for $\dist$ and $\disth$ (see \cref{sec:proof_triangle}) gives
    \begin{align}
        \dist\prn*{\hat{\bm x}_k,\mathcal X^\ast(t_{k})}
        &\leq \dist\prn*{\hat{\bm x}_k,\mathcal X^\ast(t_{k-1})}
        +\disth\prn*{\mathcal X^\ast(t_{k-1}),\mathcal X^\ast(t_{k})}\\
        &\leq \|\hat{\bm x}_k-\bm x_{k-1}\|+\dist\prn*{\bm x_{k-1},\mathcal X^\ast(t_{k-1})}
        +\disth\prn*{\mathcal X^\ast(t_{k-1}),\mathcal X^\ast(t_{k})}.\label{ineq:proof_gd1}
    \end{align}

    We will bound each term of \cref{ineq:proof_gd1} in what follows.
    The first term can be bounded by $vh$ from the acceptance condition~\cref{eq:acceptance}. 
    The second term can be bounded as follows:
    \begin{align}
        \dist\prn*{\bm x_{k-1},\mathcal X^\ast(t_{k-1})}^2
        &\leq\frac2\mu\prn[\big]{f(\bm x_{k-1};t_{k-1})-f^\ast(t_{k-1})}\nonumber\\
        &\leq\frac2\mu \theta_2^C\prn[\big]{f(\hat{\bm x}_{k-1};t_{k-1})-f^\ast(t_{k-1})}\label{ineq:proof_gd2}\\
        &\leq \frac{L_{2,0}}\mu\theta_2^C\dist\prn*{\hat{\bm x}_{k-1},\mathcal X^\ast(t_{k-1})}^2.\label{ineq:proof_gd3}
    \end{align}
    where the first inequality follows from the quadratic growth property of P\L~functions, i.e., $f(\bm x;t)-f^\ast(t)\geq(\mu/2)\dist(\bm x,\mathcal X^\ast(t))^2$ for all $\bm x\in\R^n$ and $t\geq0$ (see, e.g., \cite[Theorem 2]{karimi2016linear} for details), the second follows from the linear convergence of GD, and the last follows from the inequality~\cref{ineq:descent_lemma}. 
    \cref{lem:lipschitz} yields the following upper bound on the third term of \cref{ineq:proof_gd1}:
    \begin{gather}
        \disth\prn*{\mathcal X^\ast(t_{k-1}),\mathcal X^\ast(t_{k})}\leq\frac{L_{1,1}}\mu h.\label{ineq:proof_gd5}
    \end{gather}
    Combining the inequalities \cref{ineq:proof_gd1,ineq:proof_gd3,ineq:proof_gd5} completes the proof.
\end{proof}

By using \cref{lem:recursive2}, we get a tracking error bound for general P\L~functions.

\begin{theorem}[Tracking Error of SHARP, P\L]\label{thm:error_pl}
    Suppose that \cref{ass:smooth} and \cref{ass:pl_global} hold, and set 
    \begin{gather}
        C>\frac{\log\kappa}{\log(\theta_2^{-1})},\quad\text{where}\quad\kappa\coloneqq \frac{L_{2,0}}\mu,\quad\theta_2\coloneqq1-\alpha\mu(2-\alpha L_{2,0})<1,\label{ineq:c_pl}
    \end{gather}
    {and $\alpha$ is the step size in \cref{ass:smooth2}.}
    Then the following holds:
    \begin{gather}
        \limsup_{k\to\infty}\dist\prn*{\hat{\bm x}_k,\mathcal X^\ast(t_{k})}\leq\frac{v+\frac{L_{1,1}}\mu}{1-\sqrt{\kappa \theta_2^{C}}}h=O(h)\label{ineq:error_pl}        
    \end{gather}
    and
    \begin{gather}
        \limsup_{k\to\infty}\prn[\big]{f(\hat{\bm x}_k;t_k)-f^\ast(t_k)}\leq\frac{L_{2,0}}2\left(\frac{v+\frac{L_{1,1}}\mu}{1-\sqrt{\kappa \theta_2^{C}}}\right)^2h^2=O(h^2).\label{ineq:error_pl2}
    \end{gather}
\end{theorem}

\begin{proof}
    \cref{lem:recursive2} guarantees the inequality~\cref{ineq:recursive2} for all $k\geq 2$.
    Solving \cref{ineq:recursive2} yields
    \[
        \dist\prn*{\hat{\bm x}_k,\mathcal X^\ast(t_{k})}
        \leq\frac{v+\frac{L_{1,1}}\mu}{1-\sqrt{\kappa \theta_2^{C}}}h+\left(\dist\prn*{\hat{\bm x}_1,\mathcal X^\ast(t_{1})}-\frac{v+\frac{L_{1,1}}\mu}{1-\sqrt{\kappa \theta_2^{C}}}h\right)\left(\sqrt{\kappa \theta_2^{C}}\right)^{k-1}
    \]
    for all $k\geq1$.
    Since \cref{ineq:c_pl} implies $\sqrt{\kappa \theta_2^{C}}<1$, taking the limit $k\to\infty$ gives \cref{ineq:error_pl}.

    The inequalities~\cref{ineq:descent_lemma} and \cref{ineq:error_pl} yield \cref{ineq:error_pl2}.
\end{proof}

The derived tracking error bound is $O(h^2)$, which is better than the $O(h)$ bound for TVGD obtained by Iwakiri et al.~\cite[Theorem 3.3]{iwakiri2024prediction}.

\subsection{Case 4: Non-convex Functions}\label{sec:nonconvex}

Finally, we consider general non-convex functions.
We give a bound on {the average of} $\|\nabla_{\bm x}f(\hat{\bm x}_k;t_k)\|$ as a tracking error because it is hopeless to bound $\dist(\hat{\bm x}_k,\mathcal X^\ast(t_k))$.
Note that bounding the average of the gradient norm is a weaker requirement than bounding the distance to the set of stationary points, and we do not need to consider issues such as bifurcation of stationary points.

The following tracking error bound shows that \cref{alg:hospca} can track stationary points thanks to the acceptance condition, even if the objective function is not strongly convex or P\L.

\begin{theorem}[Tracking Error of SHARP, Non-convex]\label{thm:error_nc}
    Suppose that \cref{ass:smooth} holds, $\inf_{(\bm x,t)\in\R^n\times\R_{\geq0}}f(\bm x;t)>-\infty$, there exist constants $L_{1,0}, L_{0,1}\geq0$ such that
    \begin{align}
        |f(\bm x;t)-f(\bm y;t)|&\leq L_{1,0}\|\bm x-\bm y\|\quad\text{and}\quad
        |f(\bm x;s)-f(\bm x;t)|\leq L_{0,1}|s-t|\label{ineq:lip_nc}
    \end{align}
    for all $\bm x,\bm y\in\R^n$ and $s,t\geq0$.
    Then, the following holds for all $K\geq1$:
    \[
        \frac1K\sum_{k=1}^K\|\nabla_{\bm x}f(\hat{\bm x}_k;t_k)\|
        \leq\sqrt{\frac{L_{1,0}v+L_{0,1}}{\alpha-\frac{L_{2,0}}2\alpha^2}h}+O\left(\frac1{\sqrt K}\right),
    \]
    {where $\alpha$ is the step size in \cref{ass:smooth2}}, and hence
    \[
        \limsup_{K\to\infty}\frac1K\sum_{k=1}^K\|\nabla_{\bm x}f(\hat{\bm x}_k;t_k)\|\leq\sqrt{\frac{L_{1,0}v+L_{0,1}}{\alpha-\frac{L_{2,0}}2\alpha^2}h}=O(\sqrt h).
    \]
\end{theorem}

\begin{proof}
    {By using \cref{ineq:descent_lemma,eq:correction}, we can see that \cref{ass:smooth} guarantees 
    \begin{align}
        f(\bm x_{k}^{c+1};t_{k})
        &\leq f(\bm x_{k}^c;t_{k})+\mybracket{\nabla_{\bm x}f(\bm x_{k}^c;t_{k}),\bm x_{k}^{c+1}-\bm x_{k}^c}+\frac{L_{2,0}}2\|\bm x_{k}^{c+1}-\bm x_{k}^c\|^2\\
        &= f(\bm x_{k}^c;t_{k})-\left(\alpha-\frac{L_{2,0}}2\alpha^2\right)\|\nabla_{\bm x}f(\bm x_{k}^c;t_{k})\|^2\label{ineq:descent}
    \end{align}
    for all $0\leq c\leq C-1$.}
    Summing up {this inequality} for $c=0,\ldots,C-1$ gives
    \begin{align}
        f(\bm x_{k};t_{k})
        &\leq f(\hat{\bm x}_k;t_{k})-\left(\alpha-\frac{L_{2,0}}2\alpha^2\right)\sum_{c=0}^{C-1}\|\nabla_{\bm x}f(\bm x_{k}^c;t_{k})\|^2\nonumber\\
        &\leq f(\hat{\bm x}_k;t_{k})-\left(\alpha-\frac{L_{2,0}}2\alpha^2\right)\|\nabla_{\bm x}f(\hat{\bm x}_k;t_{k})\|^2,\label{ineq:proof_global1}
    \end{align}
    where we neglected $\sum_{c=1}^{C-1}\|\nabla_{\bm x}f(\bm x_{k}^c;t_{k})\|^2\geq0$ in the last inequality.

    The assumption~\cref{ineq:lip_nc} yields
    \begin{align}
        f(\hat{\bm x}_k;t_{k})
        &=f(\bm x_{k-1};t_{k-1})+\prn[\big]{f(\hat{\bm x}_{k};t_{k-1})-f(\bm x_{k-1};t_{k-1})}\\&\hspace{72pt}+\prn[\big]{f(\hat{\bm x}_{k};t_{k})-f(\hat{\bm x}_{k};t_{k-1})}\\
        &\leq f(\bm x_{k-1};t_{k-1})+L_{1,0}\|\hat{\bm x}_k-\bm x_{k-1}\|+L_{0,1}h\nonumber\\
        &\leq f(\bm x_{k-1};t_{k-1})+(L_{1,0}v+L_{0,1})h,\label{ineq:proof_global2}
    \end{align}
    where the last inequality follows from the acceptance condition~\cref{eq:acceptance}.

    By combining \cref{ineq:proof_global1} and \cref{ineq:proof_global2}, we get
    \[
        f(\bm x_{k};t_{k})
        \leq f(\bm x_{k-1};t_{k-1})+(L_{1,0}v+L_{0,1})h-\left(\alpha-\frac{L_{2,0}}2\alpha^2\right)\|\nabla_{\bm x}f(\hat{\bm x}_k;t_{k})\|^2.
    \]
    Summing up this inequality for $k=1,2,\ldots,K$ gives
    \[
        f(\bm x_{K};t_{K})
        \leq f(\bm x_{0};t_{0})+(L_{1,0}v+L_{0,1})Kh-\left(\alpha-\frac{L_{2,0}}2\alpha^2\right)\sum_{k=1}^K\|\nabla_{\bm x}f(\hat{\bm x}_k;t_{k})\|^2.
    \]
    Thus, we obtain
    \begin{align}
        \frac1{K}\sum_{k=1}^K\|\nabla_{\bm x}f(\hat{\bm x}_k;t_{k})\|^2
        &\leq\frac1{\alpha-\frac{L_{2,0}}2\alpha^2}\left((L_{1,0}v+L_{0,1})h+\frac{f(\bm x_{0};t_{0})-f(\bm x_{K};t_{K})}K\right)\\
        &=\frac{L_{1,0}v+L_{0,1}}{\alpha-\frac{L_{2,0}}2\alpha^2}h+O\left(\frac1K\right),
    \end{align}
    where we used $\inf_{(\bm x,t)\in\R^n\times\R_{\geq0}}f(\bm x;t)>-\infty$ in the last equality.
    Using the Cauchy--Schwarz inequality completes the proof.
\end{proof}

\section{Numerical Experiments}

This section evaluates the numerical performance of the proposed algorithm in three problem settings.
{We used GD update~\cref{eq:correction} with step size $\alpha=1/L_{2,0}$ for the correction steps, unless explicitly stated otherwise.
How to choose parameters $P$, $C$, and $v$ of \cref{alg:hospca} in general is discussed in \cref{sec:parameter}.}
{All the experiments were conducted in Python 3.13.5 on a MacBook Air whose processor is Apple M4 and memory is 32 GB.\footnote{The code is available at \url{https://github.com/TomoyaKamijima/sharp_exp}}}

\subsection{Experiment 1: Toy Problem}
To explore the relationship between parameters and accuracy, we consider the following one-dimensional toy problem: 
\[
    \min_{x\in\R}f(x;t)\coloneqq\sin(x-t)+\frac1{10}x^2.
\]
This function is non-convex in $x$ and has multiple isolated trajectories of local optima that disappear at some time.
We fixed the initial point to $x_0=0$.
Since $\nabla_{xx}f(x;t)=-\sin(x-t)+1/5$, the function $f(\cdot;t)$ is $L_{2,0}$-smooth and $f(\cdot;t)$ is locally $\mu$-strongly convex with $(L_{2,0},\mu)=(1.2,0.2)$ around the local optimum satisfying $-\sin(x-t)\geq0$.
Hence, we set $\alpha=1/1.2$ to satisfy \cref{ass:smooth2}.
For the constants $\theta_1$, $L_{1,1}$, and $\sigma_1$, we can derive $\theta_1=1-\mu/L_{2,0}=5/6$, $L_{1,1}=1$, and $\sigma_1\leq L_{1,1}/\mu=5$ in this setting.
This setting satisfies the assumptions of \cref{thm:error_nc}.
\cref{ass:trajectory_local} does not hold at times when a local solution disappears, but it does hold on the time intervals that exclude such times provided the threshold $v$ and the number $C$ of the correction step are sufficiently large.

We investigate the behavior of the tracking error of the proposed algorithm when the four parameters $h$, $P$, $C$, and $v$ are varied. 
We conducted two types of experiments: (i) fixing $(C,v)=(30,20)$ while varying $h\in\{1,0.1,0.01\}$ and $P\in\{1,2,4,7\}$, and (ii) fixing $(h,P)=(0.1,7)$ while varying $C\in\{1,5,30\}$ and $v\in\{0.1,1,20\}$.
The values of $C$ and $v$ are derived from the following observations.
To satisfy the first inequality in \cref{ineq:condition2} with $(\gamma,P)=(\theta_1^C,7)$, it suffices to take $C\geq27$.
The condition \cref{ineq:condition4} with $(p,C,\gamma,\sigma_1)=(7,30,\theta_1^C,5)$ can roughly be rewritten into $v\geq16.4$ when we neglect the terms $\gamma^{k-k_0-p+1}$ and $\sigma_{p}h^p$ as sufficiently small.

\begin{figure}[tb]
  \centering
  \includegraphics[width=0.7\textwidth]{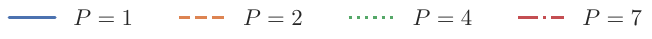}\\
  \begin{minipage}[b]{0.32\hsize}
    \centering
    \includegraphics[width=\textwidth]{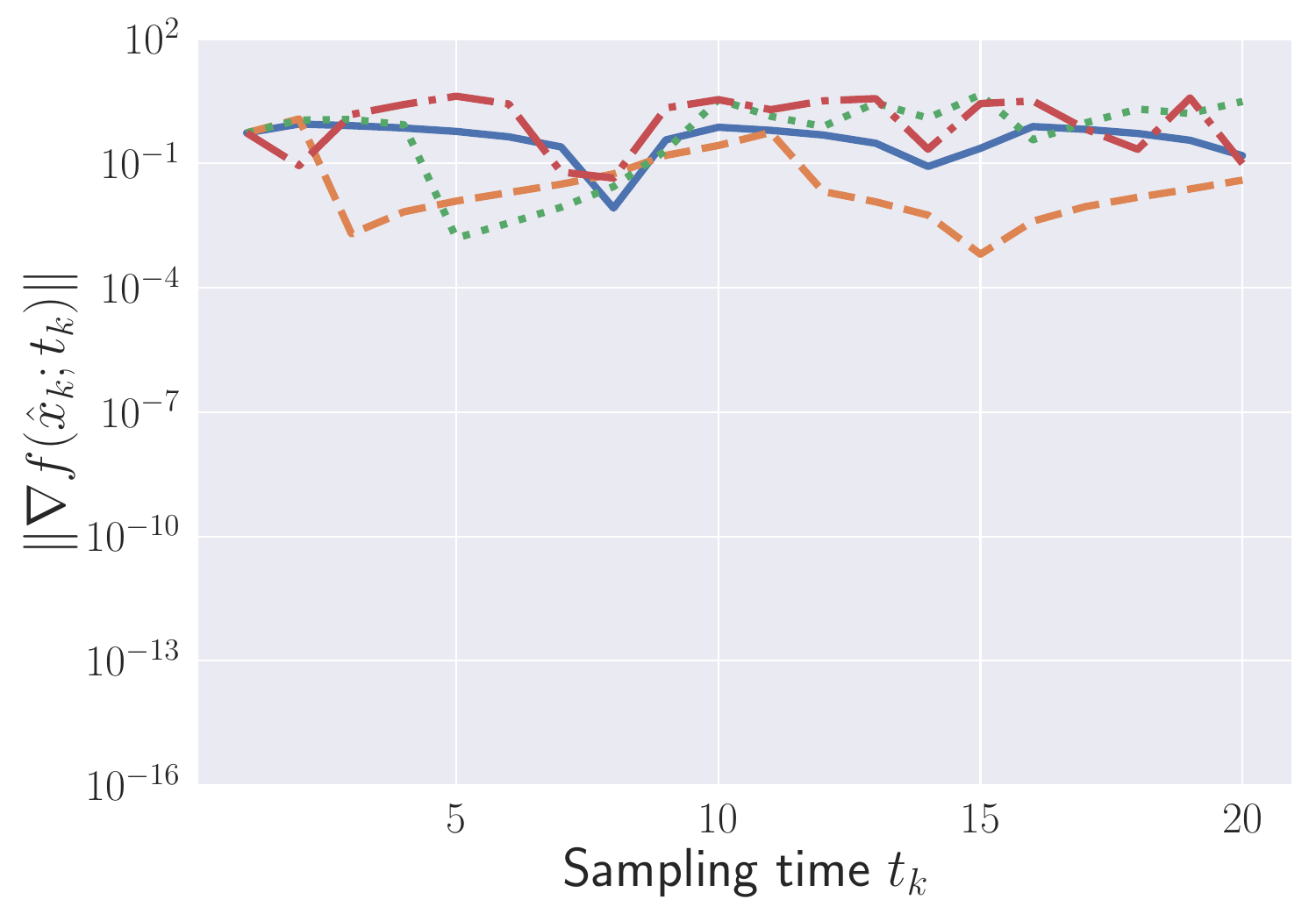}
    \subcaption{$h=1$}\label{fig:experiment11_1}
  \end{minipage}
  \begin{minipage}[b]{0.32\hsize}
    \centering
    \includegraphics[width=\textwidth]{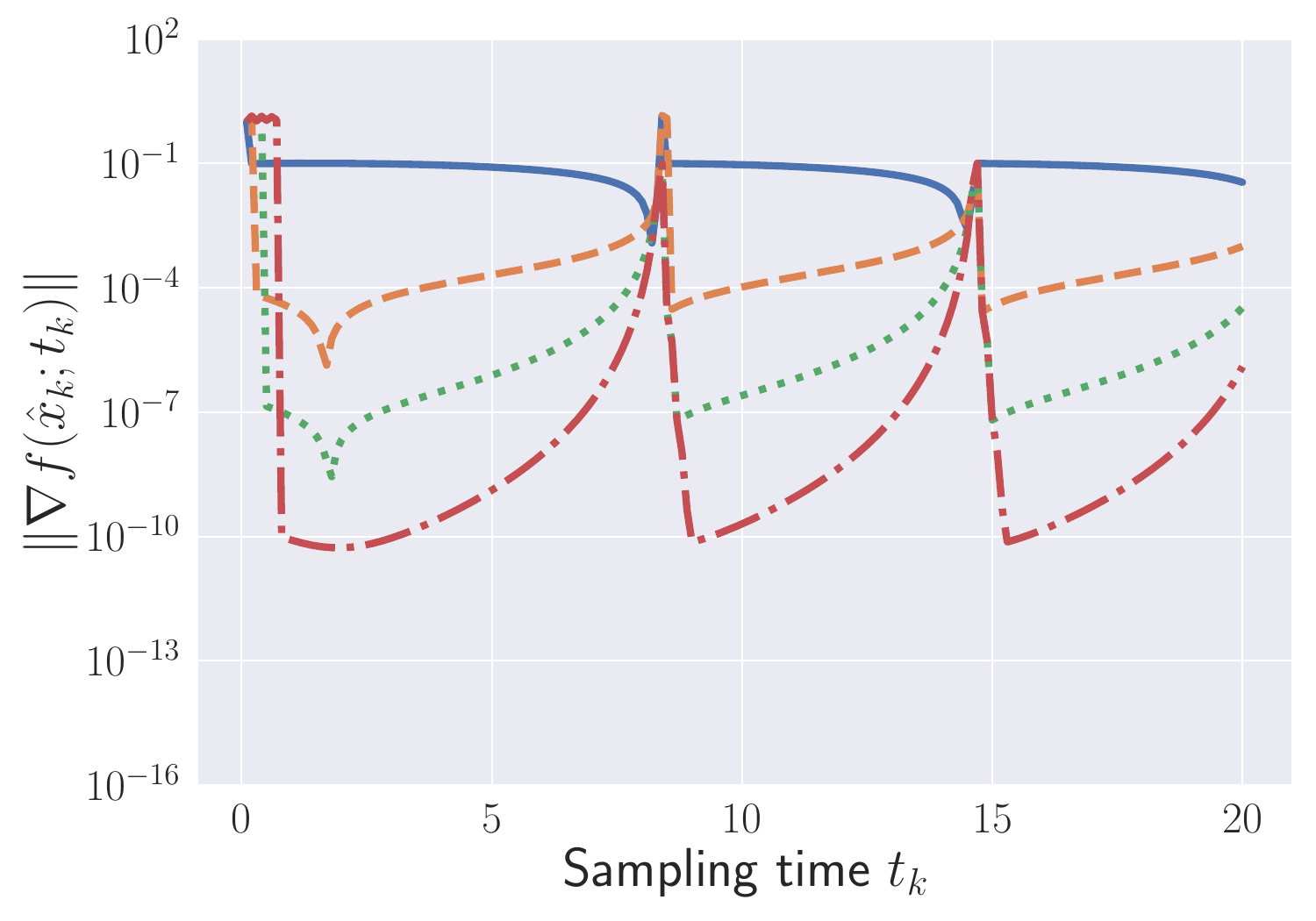}
    \subcaption{$h=0.1$}\label{fig:experiment11_2}
  \end{minipage}
  \begin{minipage}[b]{0.32\hsize}
    \centering
    \includegraphics[width=\textwidth]{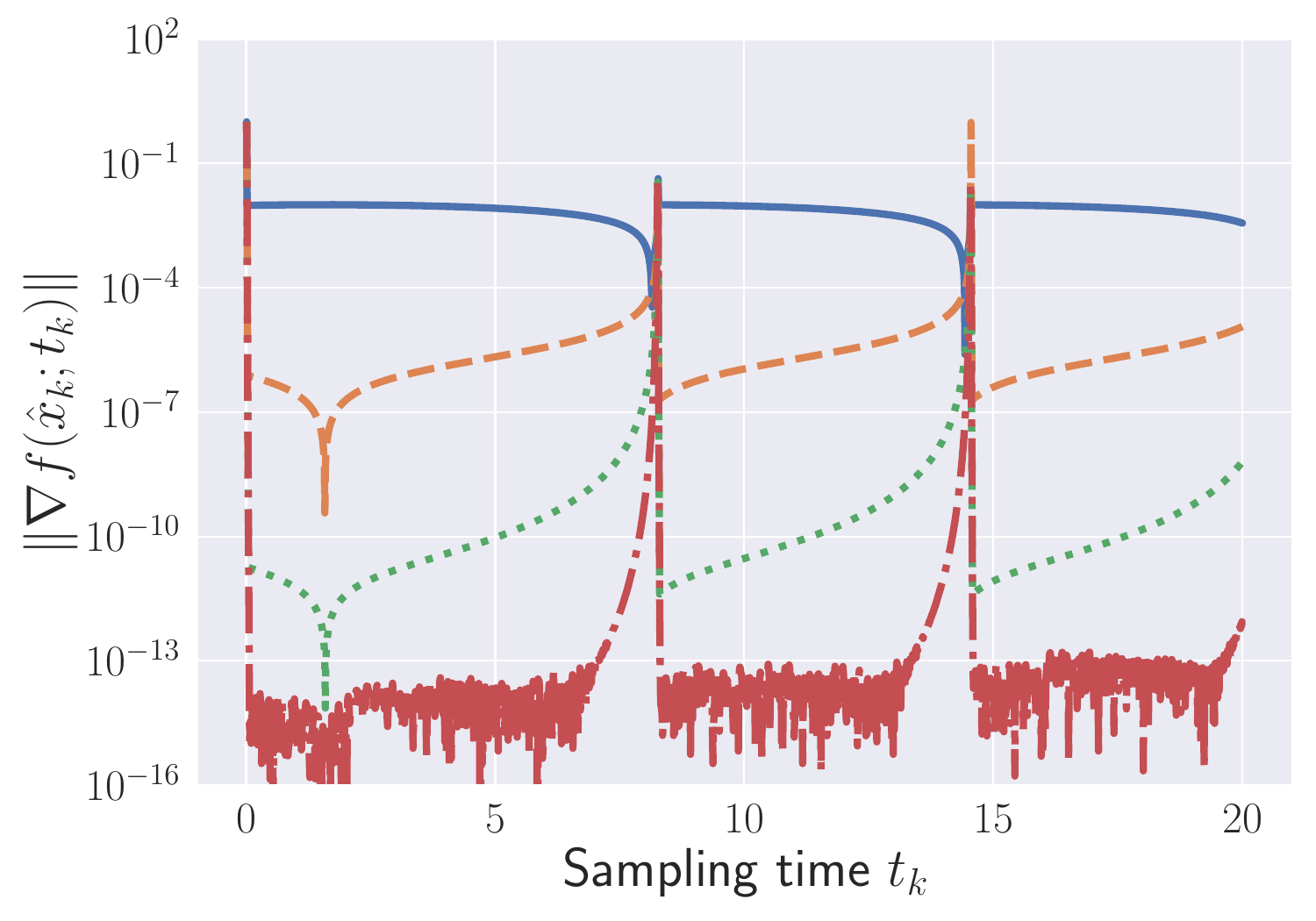}
    \subcaption{$h=0.01$}\label{fig:experiment11_3}
  \end{minipage}
  \caption{Results of Experiment 1(i)}\label{fig:experiment11}
\end{figure}
\begin{figure}[tb]
  \centering
  \includegraphics[width=0.6\textwidth]{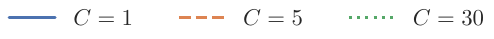}\\
  \begin{minipage}[b]{0.32\hsize}
    \centering
    \includegraphics[width=\textwidth]{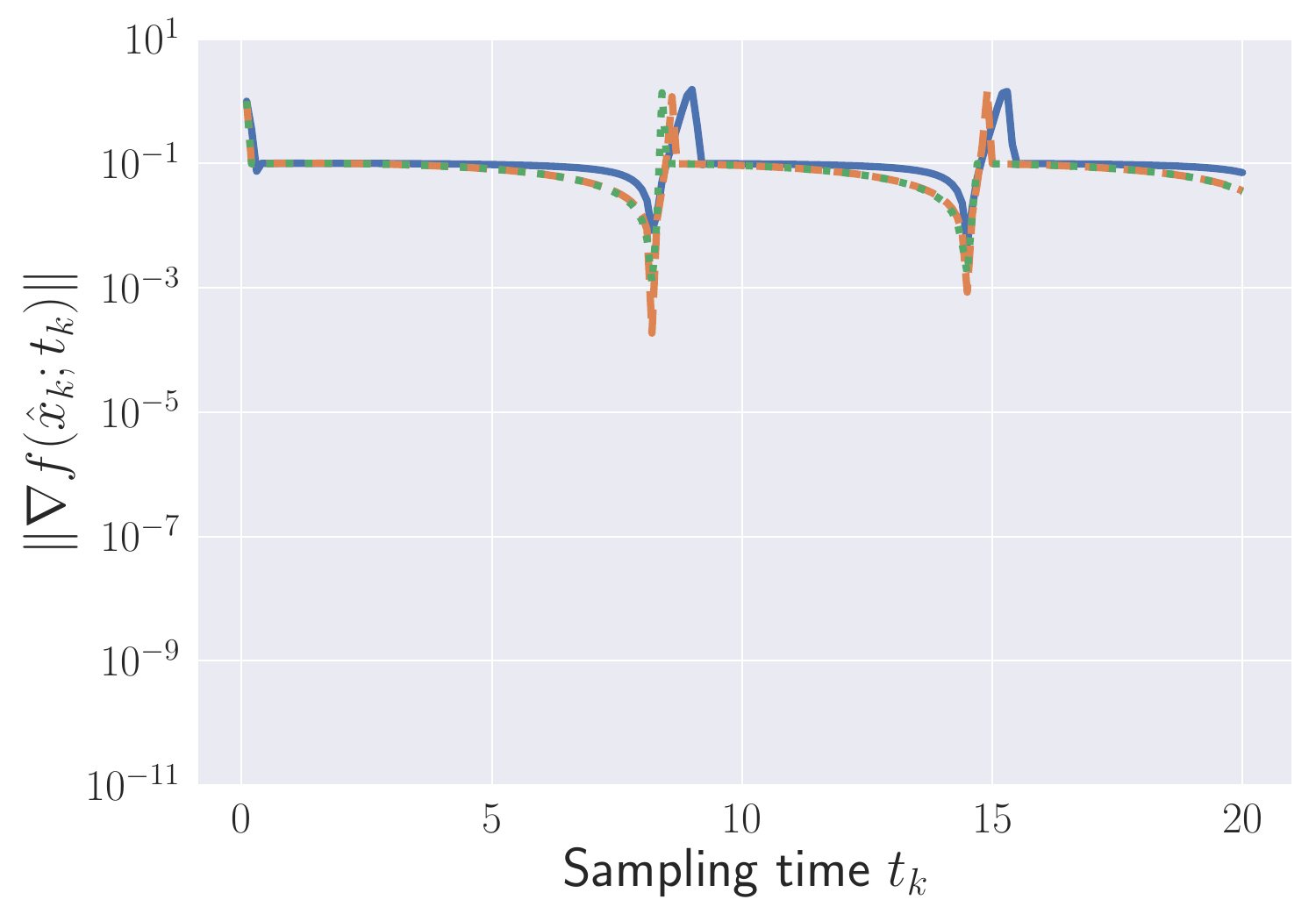}
    \subcaption{$v=0.1$}\label{fig:experiment12_1}
  \end{minipage}
  \begin{minipage}[b]{0.32\hsize}
    \centering
    \includegraphics[width=\textwidth]{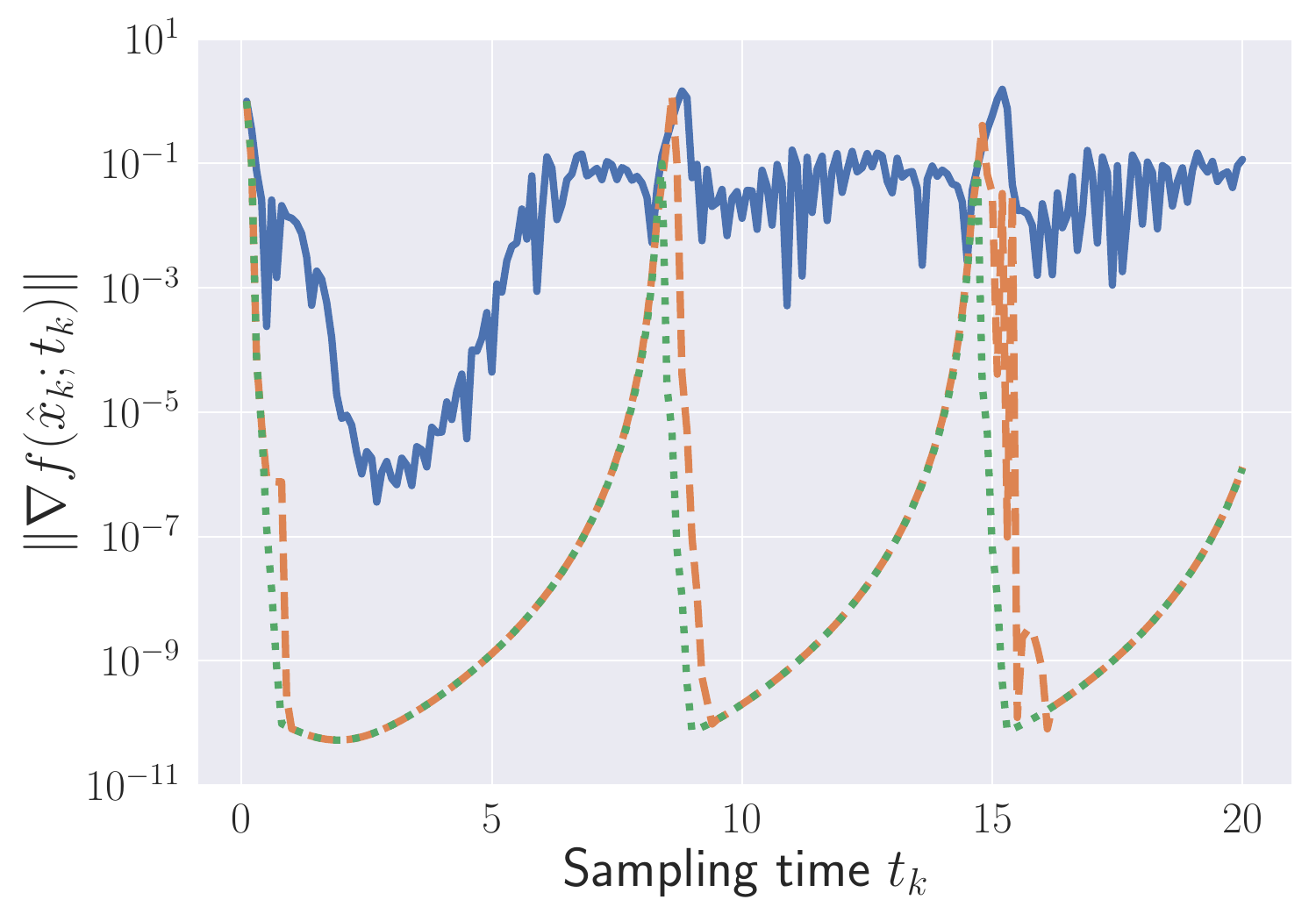}
    \subcaption{$v=1$}\label{fig:experiment12_2}
  \end{minipage}
  \begin{minipage}[b]{0.32\hsize}
    \centering
    \includegraphics[width=\textwidth]{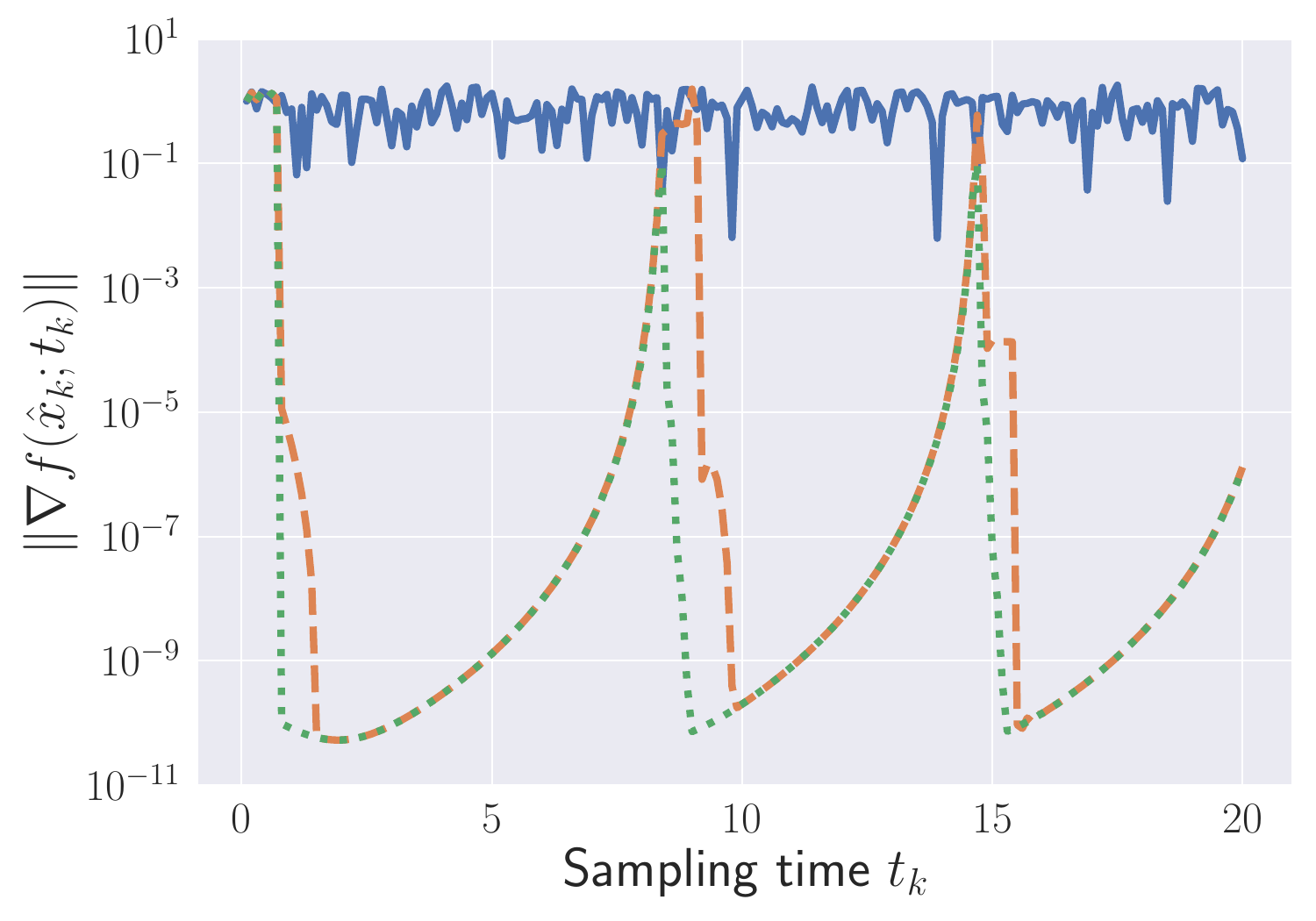}
    \subcaption{$v=20$}\label{fig:experiment12_3}
  \end{minipage}
  \caption{Results of Experiment 1(ii)}\label{fig:experiment12}
\end{figure}

The results of (i) are shown in \cref{fig:experiment11}.
We measured the tracking error by $\|\nabla_{\bm x}f(\hat{\bm x}_k;t_k)\|$ instead of $\|\hat{\bm x}_k-\bm x^\ast(t_k)\|$ because the latter one is difficult to compute.
\cref{fig:experiment11_1} ($h=1$) shows no significant difference in the tracking error with different $P$.
This is because $\sigma_{p}h^{p}$ in \cref{ineq:error_general} does not decrease as $p$ increases.
\cref{fig:experiment11_2} ($h=0.1$) shows that the error decreases as $P$ increases, as presented in \cref{ineq:error_general} with $p=P$.
Another noteworthy aspect is the behavior around $t=8$.
At this time, the target trajectory disappears, and the error increases. Still, the proposed method did not diverge and succeeded in tracking a new target trajectory thanks to the acceptance condition~\cref{eq:acceptance}.
\cref{fig:experiment11_3} ($h=0.01$) shows that the error gets much smaller than the case $h=0.1$.

The results of (ii) are shown in \cref{fig:experiment12}.
\cref{fig:experiment12_1} ($v=0.1$) shows that the tracking error was not improved by increasing $C$.
This is because the threshold $v$ is too small.
\cref{fig:experiment12_2} ($v=1$) shows that the error was lower for $C=5$ and $C=30$ than for $C=1$.
Notably, the observed values in the case of $(C,v)=(5,1)$ are smaller than the theoretically expected values $(C,v)=(27,16.4)$.
This result suggests that the conditions~\cref{ineq:condition2,ineq:condition4} in the theoretical analysis are conservative in this setting.
\cref{fig:experiment12_3} ($v=20$) shows that the error was similar to the case $v=1$ in $t\geq8$.
This result suggests that small values of $C$ cannot improve the error even when $v$ is large.
Another noteworthy observation is that the error in the case $(C,v)=(1,20)$ is larger than those in the cases $(C,v)=(1,0.1)$ and $(C,v)=(1,1)$.
This is consistent with the theoretical result in \cref{thm:error_nc} that the upper bound of the error gets worse as $v$ increases when the assumptions for the $O(h^{p})$ error are not satisfied.

\subsection{Experiment 2: {Quadratic Function}}
We consider the following strongly convex problem:
\[
    \min_{\bm x\in\R^{1000}}f(\bm x;t)\coloneqq\inner{\bm x-\bm y(t)}{\bm A(\bm x-\bm y(t))},
\]
where $\bm y(t)=(10\cos\frac t{1000},10\cos\frac {2t}{1000},\ldots,10\cos t)$ is a target trajectory and $\bm A$ is a positive definite matrix sampled randomly via its diagonalization so that its minimum eigenvalue is $\mu=1$ and maximum is $L_{2,0}=10$.

We compared the following algorithms:
\begin{itemize}[leftmargin=*]
  \item TVGD~\cite{popkov05}.
  \item GTT and NTT~\cite{simonetto16}. The correction step of NTT is a single Newton step, which gives the exact solution because $f$ is quadratic.
  \item DTZD~\cite{qiu2019new}. The correction step is a single Newton step.
  \item EPC~\cite{bastianello23}. The prediction step minimizes the extrapolated \emph{function} using $f(\cdot;t_{k-1}),\ldots,f(\cdot;t_{k-P})$. We set $P=1,2,\ldots,9$. The minimization is performed by $C$ steps of GD with step size $\alpha=1/L_{2,0}$. 
  \item Proposed method (SHARP). We set $P=1,2,\ldots,9$. 
\end{itemize}
The computational costs and theoretical error bounds of these algorithms are summarized in \cref{tab:experiments}.

We fixed the sampling period to $h=0.1$ and the initial point to $\bm x_0=(0,0)$.
The constants $\mu$, $L_{2,0}$, and $L_{1,1}$ can be computed as {$\mu=1$, $L_{2,0}=10$, and $L_{1,1}\leq\|\bm A\|_{\mathrm{op}}\|\dot{\bm y}(t)\|\leq1828$}.
We can derive {$\theta_1=1-\mu/L_{2,0}=0.9$ and $\sigma_1\leq L_{1,1}/\mu=1828$}, where $\theta_1$ and $\sigma_1$ are the constants in {\cref{sec:sc}}.
To satisfy the first inequality in \cref{ineq:condition2} with $(\gamma,P)=(\theta_1^C,{9})$, it suffices to take {$C\geq60$}.
On the other hand, EPC requires $\theta_1^C=O(h^{P-1})$ to achieve an $O(h^P)$ tracking error \cite[Theorem 1]{bastianello23}.
Hence, we fixed $C=200$ to satisfy $\theta_1^C\leq h^{P-1}$ even when $P=9$.
The condition~\cref{ineq:condition4} with $(p,C,\gamma)=({9,200},\theta_1^C)$ 
reduced to $v\geq{(1+10^{-6})}\sigma_1$ when we neglect the terms $\gamma^{k-k_0-p+1}$ and $\sigma_{p}h^p$ as sufficiently small.
For this reason, we set $v=2000$, since $2000>{(1+10^{-6})\cdot1828}$.

\begin{table}[tb]
  \centering
  \caption{Algorithms and their theoretical error bounds in Experiment 2. 
  The ``Cost per round'' column shows the computation cost to achieve the error bound. 
  Each $O(h^p)$ error bound assumes $\sigma_p<\infty$. Note that the error bound for NTT is improved to $O(h^4)$ after the correction step \cite{simonetto16}, but we focus on the error before the correction step for comparison.}
  \label{tab:experiments}
  \begin{tabular}{lcc}
    \toprule
    Algorithm & Cost per round & Theoretical error bound \\
    \midrule
    TVGD~\cite{popkov05} & $O(n)$ & $O(h)$ \\
    GTT~\cite{simonetto16} & $O(n^3(1+\kappa h))$ & $O(h^2)$ \\
    NTT~\cite{simonetto16} & $O(n^3)$ & $O(h^2)$ \\
    DTZD~\cite{qiu2019new} & $O(n^3)$ & $O(h^4)$ \\
    EPC~\cite{bastianello23} & $O(n\kappa P^2\log\frac1h)$ & $O(h^P)$ \\
    SHARP & $O(n\kappa P)$ &  $O(h^P)$ \\
    \bottomrule
  \end{tabular}
\end{table}

\cref{fig:experiment2} shows the results.
The left figure shows the performance of the proposed methods.
We only show the results for $P=1,3,5,7,9$ to avoid clutter.
The tracking error decreased as $P$ increased, which is consistent with \cref{thm:error_sc}.
The right figure compares the proposed method with existing algorithms.
We picked up only $P=9$ for EPC and SHARP because they achieved the lowest errors when $P=9$.
The EPC algorithm and the proposed method outperformed other methods in terms of the tracking error.
It is worth noting that the proposed method achieved a similar error to EPC, while the proposed method requires less computational cost than EPC as shown in \cref{tab:experiments} and does not need to access gradients in the past rounds.

\begin{figure}[tb]
  \begin{minipage}[t]{0.49\hsize}
    \centering
    \includegraphics[width=\textwidth]{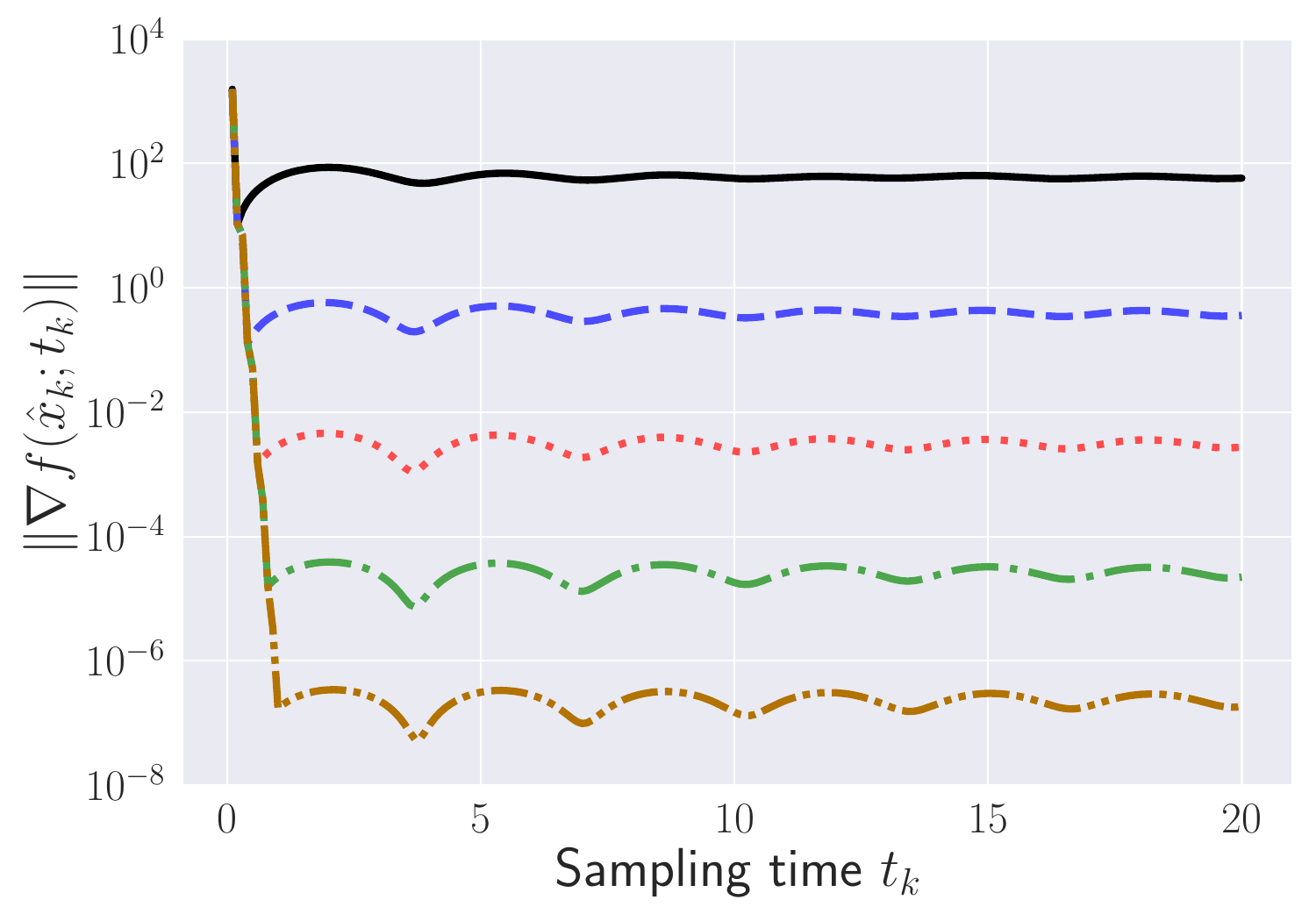}\\
    \includegraphics[width=\textwidth]{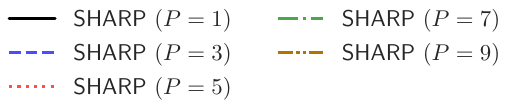}
  \end{minipage}
  \begin{minipage}[t]{0.49\hsize}
    \centering
    \includegraphics[width=\textwidth]{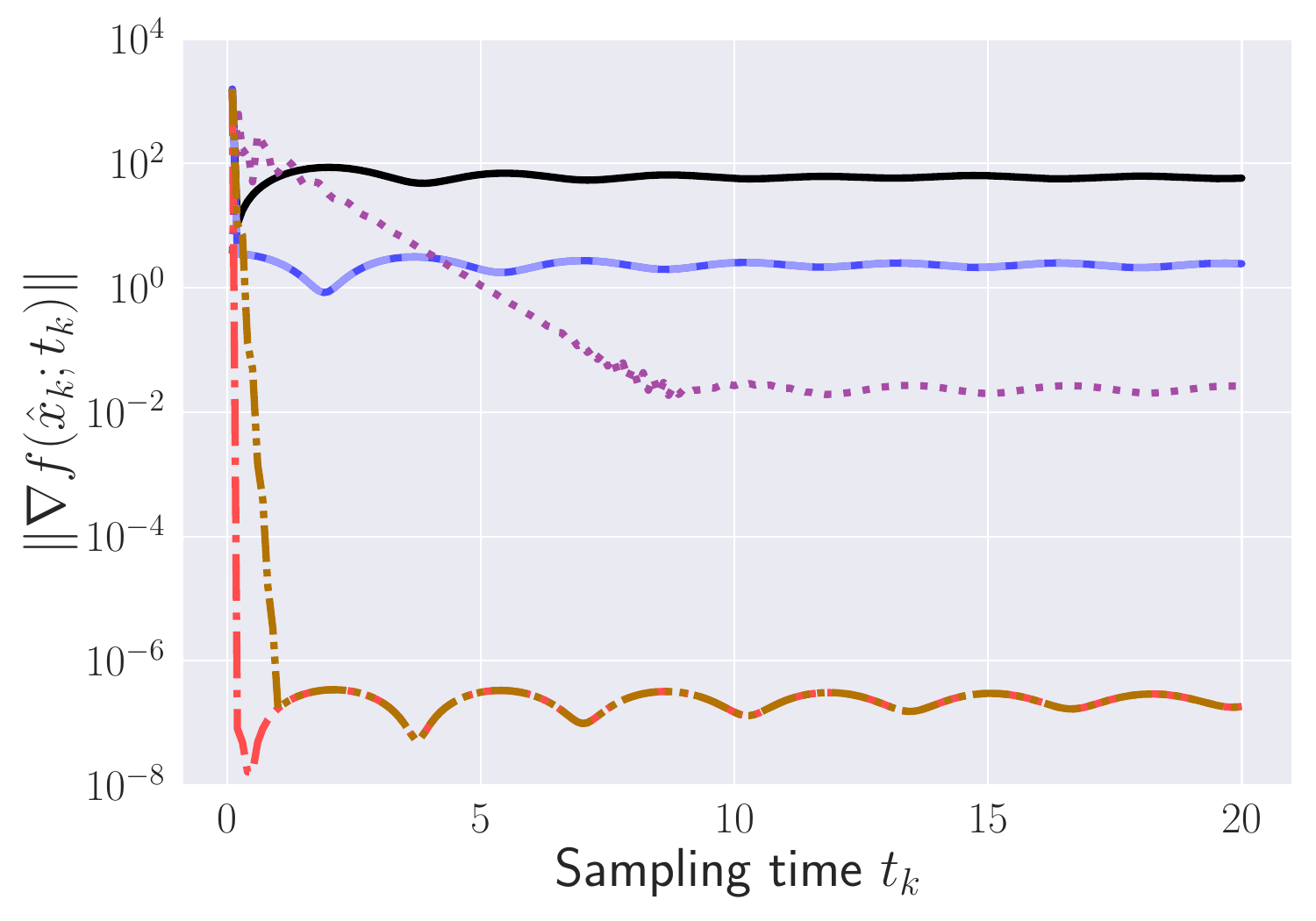}\\
    \includegraphics[width=0.85\textwidth]{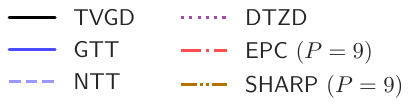}
  \end{minipage}
  \caption{Results of Experiment 2. Overlap occurs between the lines for GTT and NTT in the right figure.}\label{fig:experiment2}
\end{figure}

\subsection{Experiment 3: Non-convex Robust Regression}
We consider the following non-convex problem:
\[
    \min_{\bm x\in\R^n}f(\bm x;t)\coloneqq\frac1m\sum_{i=1}^m\ell\prn*{\bm a_i(t)^\top\bm x-b_i(t)},
\]
where $\ell(\cdot)$ is the Geman--McClure loss function~\cite{geman1985bayesian} defined by $\ell(z)\coloneqq z^2/(1+z^2)$.
{We sampled $\bm a_1(0),\ldots,\bm a_m(0)\in\R^n$ from the uniform distribution on the unit sphere and computed $\bm a_i(t)$ by $\bm a_i(t)=\exp(t\bm S)\bm a_i(0)$, where $\bm S\in\R^{n\times n}$ is a skew-symmetric matrix sampled randomly, so that $\|\bm a_i(t)\|=\|\bm a_i(0)\|=1$ for all $t\geq0$ and thus $f(\cdot;t)$ is $2$-smooth.}
{Using these vectors, $b_i(t)$ is computed by} $b_i(t)=\bm a_i(t)^\top\bm x^\ast(t)$, where $\bm x^\ast(t)=(\cos (t/n),\cos(2t/n),\ldots,\cos t)$ is a smoothly changing true trajectory.
{Since $f$ is non-convex in $\bm x$, using $(\nabla_{\bm x\bm x}f(\bm x;t))^{-1}$ is unstable in this setting, but $f$ is smooth in $\bm x$ and $t$, which implies that $f$ satisfies the assumptions in \cref{thm:error_nc}.
In addition, when $z\approx0$, the loss function $\ell(z)\approx z^2$, which implies that the objective function $f(\bm x;t)$ is locally P\L~function around the true trajectory $\bm x^\ast(t)$.
This situation corresponds to Case 2 in \cref{sec:pl_local}.}

We compared the algorithms used in Experiment 2, except for NTT and DTZD.
This is because NTT and DTZD stopped due to the singular Hessian (GTT luckily did not stop in this setting).
Note that GTT and EPC require strong convexity for their theoretical guarantees, which is not satisfied in this setting.

We fixed the sampling period to $h=0.1$ and the initial point to $\bm x_0=\bm0$.
We set the problem's parameters as {$(n,m)=(100,1000)$}.
We can compute $\sigma_1=\sup_{t\geq0}\|\dot{\bm x}^*(t)\|\leq6$ in this setting.
To satisfy the second inequality in \cref{ineq:condition4} with sufficiently small $\gamma$, we set $v=10>\sigma_1$.
The number $C$ of correction steps was chosen as large as possible so that the computation finishes within the sampling period.
In this experiment, we set $C=1000$.

The results are shown in \cref{fig:experiment3}.
The left figure shows the performance of the proposed method with $P=1,3,5,7,8,9$.
The best performance was achieved when $P=7$.
This result shows that increasing $P$ reduces the tracking error; however, it also raises the required number $C$ of correction steps to satisfy the condition \cref{ineq:condition2}, and when that requirement is no longer satisfied, the error grows.
The right figure compares the proposed method with existing algorithms.
The proposed method outperformed {existing methods} in terms of the tracking error 
when $f(\bm x;t)$ is non-convex in $\bm x$. 
This is consistent with the theoretical result in \cref{thm:error_general}, which only assumes the smoothness of the target trajectory (\cref{ass:trajectory_local2}) and local linear convergence of the correction step (\cref{ass:correction_local}).

\begin{figure}[tb]
  \begin{minipage}[t]{0.49\hsize}
    \centering
    \includegraphics[width=\textwidth]{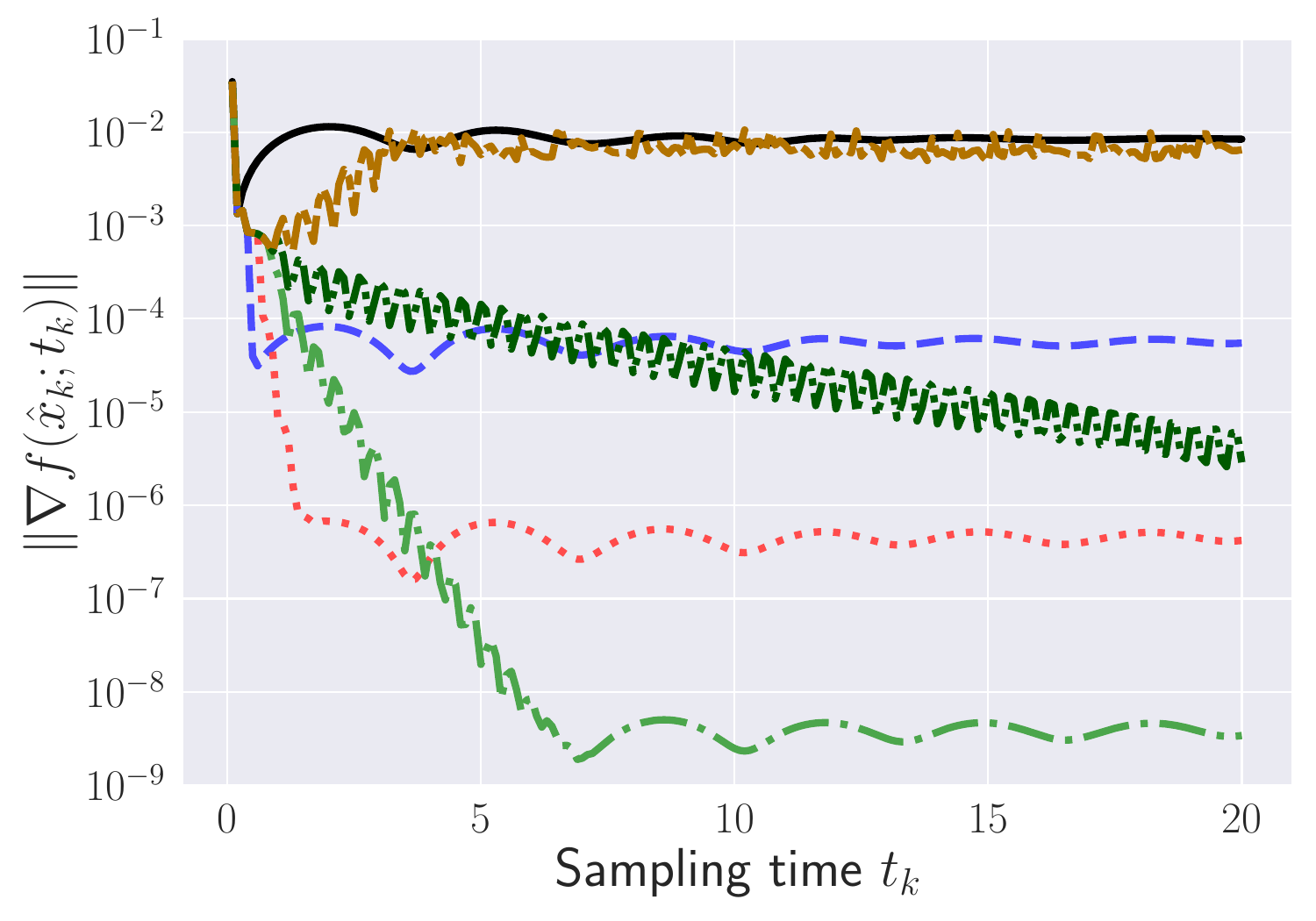}\\
    \includegraphics[width=\textwidth]{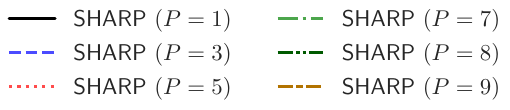}
  \end{minipage}
  \begin{minipage}[t]{0.49\hsize}
    \centering
    \includegraphics[width=\textwidth]{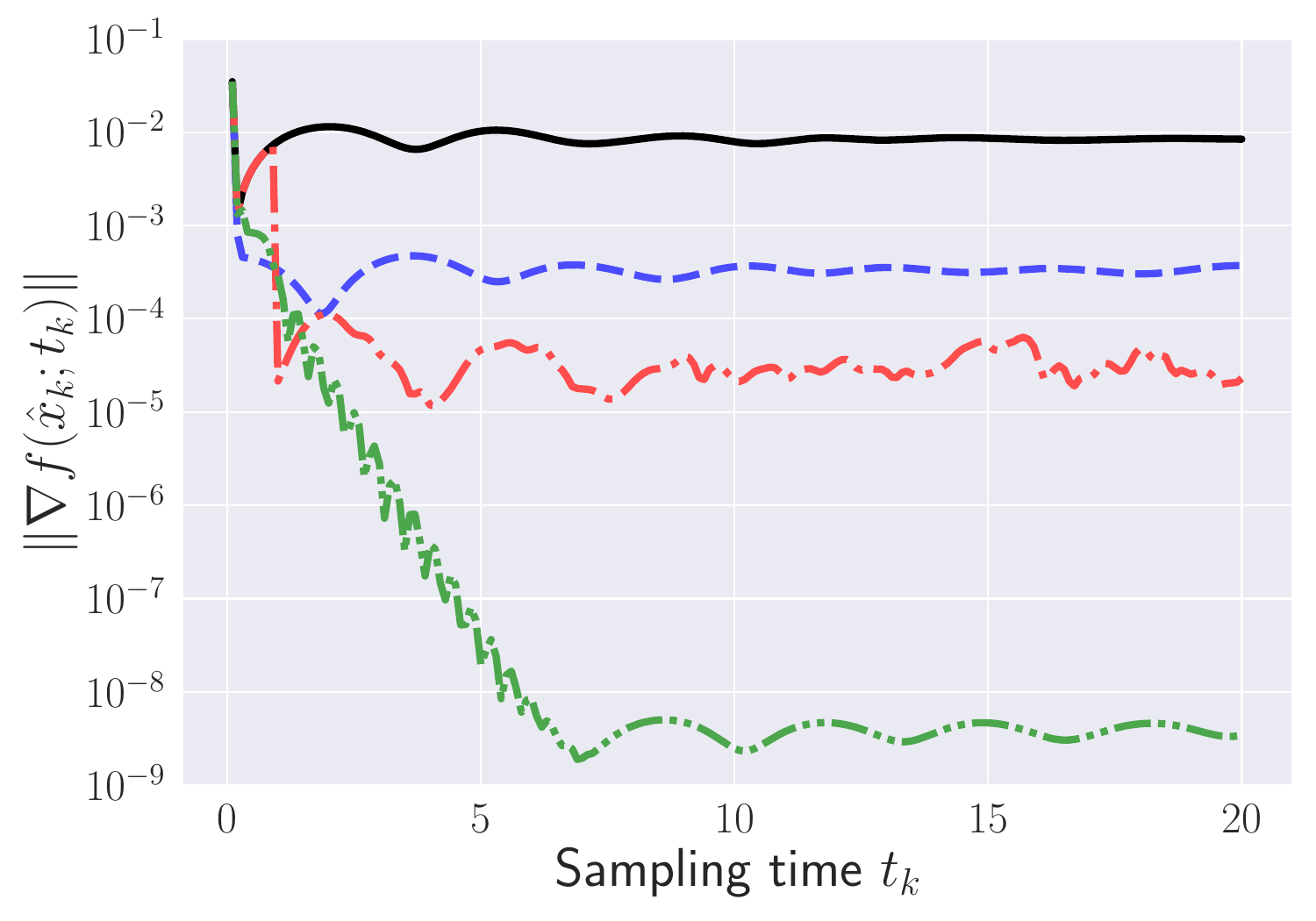}\\
    \includegraphics[width=0.85\textwidth]{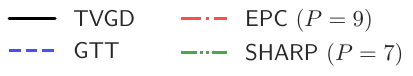}
  \end{minipage}
  \caption{Results of Experiment 3}\label{fig:experiment3}
\end{figure}

\section{Conclusion}

We proposed a prediction-correction algorithm, SHARP, which tracks a target trajectory of time-varying optimization problems.
Its prediction scheme is based on the Lagrange interpolation of past solutions and an acceptance condition, which can be computed without Hessian matrices or even gradients.
We proved that the proposed method achieves an $O(h^{p})$ tracking error, assuming that {the $p$-th derivative of the target trajectory is bounded} and the correction step converges locally linearly.
We also proved that the proposed method can track a trajectory of stationary points even if the objective function is non-convex.
Numerical experiments demonstrated that our algorithm tracks a target trajectory with high accuracy.

In the future, we will investigate how the proposed method can be generalized to more complex problems, such as constrained, non-smooth, or stochastic problems.
Other possible directions are to weaken the assumptions of the tracking error analysis {and make the algorithm free from hyperparameter tuning.}

\appendix

\section{Proof of \cref{lem:velocity1}}\label{sec:proof_velocity1}

\begin{proof}
    First, we prove
    \begin{gather}
        \int_{[0,1]^p}\bm \varphi^{(p)}\left(\sum_{i=1}^ps_i\right)\dd \bm s=\sum_{i=0}^p(-1)^{p-i}\binom pi\bm \varphi(i)\label{eq:integral}
    \end{gather}
    for any $p$-times differentiable function $\bm \varphi\colon\R\to\R^n$ by induction.
    For $p=1$, the equation~\cref{eq:integral} is equivalent to $\int_0^1\bm \varphi'(s_1)\dd s_1=\bm \varphi(1)-\bm \varphi(0)$,
    which holds from the fundamental theorem of calculus.
    Assume \cref{eq:integral} holds for some $p$.
    Then, we have
    \begin{align}
        \int_{[0,1]^{p+1}}\bm \varphi^{(p+1)}\left(\sum_{i=1}^{p+1}s_i\right)\dd \bm s
        &=\int_{[0,1]^{p}}\bm \varphi^{(p)}\left(1+\sum_{i=1}^{p}s_i\right)\dd \bm s-\int_{[0,1]^{p}}\bm \varphi^{(p)}\left(\sum_{i=1}^{p}s_i\right)\dd \bm s\\
        &=\sum_{i=0}^{p}(-1)^{p-i}\binom {p}i\bm \varphi(1+i)-\sum_{i=0}^{p}(-1)^{p-i}\binom {p}i\bm \varphi(i)\\
        &=\sum_{i=1}^{p+1}(-1)^{p-i+1}\binom {p}{i-1}\bm \varphi(i)+\sum_{i=0}^{p}(-1)^{p-i+1}\binom {p}i\bm \varphi(i)\\
        &=\sum_{i=0}^{p+1}(-1)^{p-i+1}\binom {p+1}{i}\bm \varphi(i),
    \end{align}
    where the first equality is the integration by $s_{p+1}$, and the second equality follows from the induction hypothesis.
    Thus, \cref{eq:integral} holds for any $p\geq1$.

    \cref{eq:integral} with $\bm \varphi(s)=\bm x^\ast((k-s)h)$ yields
    \begin{align}
        \left\|\bm x^\ast(t_k)-\sum_{i=1}^{p}(-1)^{i-1}\binom {p}{i}\bm x^\ast(t_{k-i})\right\|
        &=\left\|\sum_{i=0}^{p}(-1)^{m-i}\binom {p}{i}\bm x^\ast(t_{k-i})\right\|\\
        &=\left\|\int_{[0,1]^{p}}(\bm x^\ast)^{(p)}\left(\prn[\bigg]{k-\sum_{i=1}^{p}s_i}h\right)h^{p}\dd \bm s\right\|\nonumber\\
        &\leq h^{p}\sup_{t\in[t_{\underline k},t_{\overline k}]}\big\|(\bm x^\ast)^{(p)}(t)\big\|,
    \end{align}
    which completes the proof.
\end{proof}

\section{Proof of \cref{thm:error_sc,thm:error_general}}\label{sec:proof_recursive0}

{This section proves \cref{thm:error_general}.
The proof of \cref{thm:error_sc} is a special case of this proof by setting $(\underline k,\overline k)=(0,\infty)$.}

\begin{proof}
    \cref{lem:recursive_general2} guarantees \cref{ineq:recursion} for all $k_1+P\leq k\leq\overline k$.

    Let 
    \begin{gather}
        a_k\coloneqq
        \begin{dcases*}
            e_k-\frac{2^{P-p}\sigma_p}{1-(2^P-1)\gamma}h^p&if $k_1\leq k\leq k_1+P-1$,\\
            \gamma\sum_{i=1}^{P}\binom {P}{i}a_{k-i}&if $k_1+P\leq k\leq \overline k$.
        \end{dcases*}\label{eq:sequence_definition}
    \end{gather}
    We prove 
    \begin{gather}
    e_k\leq a_k+\frac{2^{P-p}\sigma_p}{1-(2^P-1)\gamma}h^p\label{ineq:sequence_induction}
    \end{gather}
    for all $k_1\leq k\leq\overline k$ by induction.
    For $k_1\leq k\leq k_1+P-1$, the inequality~\cref{ineq:sequence_induction} holds with equality from \cref{eq:sequence_definition}.
    Assume that there exists $k_2\in\{k_1+P-1,\ldots,\overline k-1\}$ such that \cref{ineq:sequence_induction} holds for all $k_1\leq k\leq k_2$.
    Then, we have
    \begin{align}
        e_{k_2+1}
        &\leq\gamma\sum_{i=1}^{P}\binom {P}{i}e_{k_2+1-i}+2^{P-p}\sigma_{p}h^{p}\\
        &\leq\gamma\sum_{i=1}^{P}\binom {P}{i}\left(a_{k_2+1-i}+\frac{2^{P-p}\sigma_p}{1-(2^P-1)\gamma}h^p\right)+2^{P-p}\sigma_{p}h^{p}\\
        &=a_{k_2+1}+\frac{2^{P-p}\sigma_p}{1-(2^P-1)\gamma}h^p,
    \end{align}
    where the first inequality follows from \cref{ineq:recursion} and the second inequality from the induction hypothesis.
    Note that the last equality is obtained from \cref{eq:sequence_definition}.
    This implies that \cref{ineq:sequence_induction} holds for $k=k_2+1$.
    Therefore, \cref{ineq:sequence_induction} holds for all $k_1\leq k\leq\overline k$.

    Next, we find an upper bound on $a_k$.
    The roots of the polynomial 
    \[
        z^{P}-\gamma\sum_{i=1}^{P}\binom{P}{i}z^{P-i}
        =(1+\gamma)z^{P}-\gamma(z+1)^{P}
    \]
    are 
    \[
        z=\prn*{(1+\gamma^{-1})^{1/P}\exp\left(\frac{2\pi i\sqrt{-1}}{P}\right)-1}^{-1}\quad\text{for}\quad i=1,2,\ldots,P,
    \]
    which are distinct from each other.
    The root with the largest absolute value is $((1+\gamma^{-1})^{1/P}-1)^{-1}$.
    Hence, there exists a constant $M\geq0$ which is independent of $\overline k$ such that the following holds for all $k_1\leq k\leq\overline k$:
    \begin{gather}
        a_k\leq M\prn*{(1+\gamma^{-1})^{1/P}-1}^{-k}.\label{ineq:proof_general11}
    \end{gather}
    By combining with \cref{ineq:sequence_induction}, we get \cref{ineq:error_general0} for all $k_1\leq k\leq\overline k$.

    Since the first inequality of \cref{ineq:condition2} guarantees $((1+\gamma^{-1})^{1/P}-1)^{-1}<1$, the asymptotic tracking error is directly obtained by taking the limit $k\to\infty$.
\end{proof}

\section{Triangle Inequalities between Point and Set}\label{sec:proof_triangle}

The following triangle inequalities hold for the distance between points and sets, which are used in the proof of \cref{lem:recursive2}.

\begin{lemma}\label{prop:triangle}
    For any $\bm x,\bm y\in\R^n$ and $\mathcal X,\mathcal Y\subset\R^n$, the following inequalities hold:
    \begin{align}
        \dist(\bm x,\mathcal Y)&\leq\|\bm x-\bm y\|+\dist(\bm y,\mathcal Y),\label{ineq:triangle1}\\
        \dist(\bm x,\mathcal Y)&\leq\dist(\bm x,\mathcal X)+\disth(\mathcal X,\mathcal Y).\label{ineq:triangle2}
    \end{align}
\end{lemma}

\begin{proof}
    The inequality~\cref{ineq:triangle1} can be proved by the triangle inequality:
    \[
        \dist(\bm x,\mathcal Y)
        =\inf_{\bm y'\in \mathcal Y}\|\bm x-\bm y'\|
        \leq\|\bm x-\bm y\|+\inf_{\bm y'\in \mathcal Y}\|\bm y-\bm y'\|
        =\|\bm x-\bm y\|+\dist(\bm y,\mathcal Y).
    \]

    The inequality~\cref{ineq:triangle1} holds even when restricting $\bm y$ to $\mathcal X$.
    In this case, the second term on the right-hand side of \cref{ineq:triangle1} can be upper bounded by $\disth(\mathcal X,\mathcal Y)$ by the definition of the Hausdorff distance.
    Hence, we have
    \[
        \dist(\bm x,\mathcal Y)
        \leq \|\bm x-\bm y\|+\disth(\mathcal X,\mathcal Y),
    \]
    for all $\bm y\in\mathcal X$.
    Taking the infimum with regard to $\bm y\in\mathcal X$ yields \cref{ineq:triangle2}.
\end{proof}

\section{Parameter Settings}\label{sec:parameter}
We explain how to choose the parameters $P$, $C$, and $v$ of \cref{alg:hospca}.

As shown in \cref{sec:tracking_error}, increasing $P$ improves the theoretical upper bound on the tracking error provided the correction step converges locally linearly and the number $C$ of iterations in the correction step is sufficiently large.
When $f(\cdot;t)$ is strongly convex for all $t\geq0$ and the condition number $\kappa$ is known, we choose $(C,P)$ so that the required computation finishes within the sampling period $h$ and \cref{ineq:c_lower} is satisfied when the correction step is implemented by GD.
In this setting, the method can achieve an $O(h^{p})$ tracking error asymptotically even with $v=\infty$; nevertheless, choosing a finite $v$ according to \cref{ineq:condition4} can improve early-iteration performance.

When $f(\cdot;t)$ is non-convex or $\kappa$ is unknown, we recommend setting $C$ as large as the computational budget allows and tuning $P$ and $v$ empirically.
When a lower bound on $\kappa$ or an upper bound of $\sigma_1$ is available, $(C,P)$ or $v$ can be chosen using \cref{ineq:c_lower} or \cref{ineq:condition4}, respectively.
Note that there is no need to consider $\sigma_p$ for $p\geq2$, because the asymptotic error bound \cref{ineq:error_general} holds for all $p\leq P$.

\section*{Acknowledgments}
This work was partially supported by JSPS KAKENHI (23H03351, 24K23853, and 25KJ1123) and JST CREST (JPMJCR24Q2).

\bibliographystyle{siamplain}
\bibliography{ref}

@article{simonetto16,
  title={A class of prediction-correction methods for time-varying convex optimization},
  author={Simonetto, Andrea and Mokhtari, Aryan and Koppel, Alec and Leus, Geert and Ribeiro, Alejandro},
  journal={IEEE Transactions on Signal Processing},
  volume={64},
  number={17},
  pages={4576--4591},
  year={2016},
  publisher={IEEE}
}

@article{bastianello23,
  title={Extrapolation-Based Prediction-Correction Methods for Time-varying Convex Optimization},
  author={Bastianello, Nicola and Carli, Ruggero and Simonetto, Andrea},
  journal={Signal Processing},
  volume={210},
  pages={109089},
  year={2023},
  publisher={Elsevier}
}

@inproceedings{bastianello19,
  title={Prediction-correction splittings for nonsmooth time-varying optimization},
  author={Bastianello, Nicola and Simonetto, Andrea and Carli, Ruggero},
  booktitle={2019 18th European Control Conference (ECC)},
  pages={1963--1968},
  year={2019},
  organization={IEEE}
}

@article{simonetto18,
  title={Dual prediction-correction methods for linearly constrained time-varying convex programs},
  author={Simonetto, Andrea},
  journal={IEEE Transactions on Automatic Control},
  volume={64},
  number={8},
  pages={3355--3361},
  year={2018},
  publisher={IEEE}
}

@article{simonetto17,
  title={Prediction-correction algorithms for time-varying constrained optimization},
  author={Simonetto, Andrea and Dall'Anese, Emiliano},
  journal={IEEE Transactions on Signal Processing},
  volume={65},
  number={20},
  pages={5481--5494},
  year={2017},
  publisher={IEEE}
}

@article{popkov05,
  title={Gradient methods for nonstationary unconstrained optimization problems},
  author={Popkov, A Yu},
  journal={Automation and Remote Control},
  volume={66},
  pages={883--891},
  year={2005},
  publisher={Springer}
}

@article{iwakiri2024prediction,
  title={Prediction-Correction Algorithm for Time-Varying Smooth Non-Convex Optimization},
  author={Iwakiri, Hidenori and Kamijima, Tomoya and Ito, Shinji and Takeda, Akiko},
  journal={arXiv preprint arXiv:2402.06181},
  year={2024}
}

@inproceedings{lin2019simplified,
  title={A simplified prediction-correction algorithm for time-varying convex optimization},
  author={Lin, Zhicong and Chen, Fei and Xiang, Linying and Guo, Ge},
  booktitle={2019 Chinese Control Conference (CCC)},
  pages={1989--1994},
  year={2019},
  organization={IEEE}
}

@article{polyak1963gradient,
title = {Gradient methods for the minimisation of functionals},
journal = {USSR Computational Mathematics and Mathematical Physics},
volume = {3},
number = {4},
pages = {864-878},
year = {1963},
issn = {0041-5553},
author = {B.T. Polyak}
}

@inproceedings{karimi2016linear,
  title={Linear convergence of gradient and proximal-gradient methods under the {Polyak}-{{\L}ojasiewicz} condition},
  author={Karimi, Hamed and Nutini, Julie and Schmidt, Mark},
  booktitle={Machine Learning and Knowledge Discovery in Databases: European Conference, ECML PKDD 2016, Riva del Garda, Italy, September 19-23, 2016, Proceedings, Part I 16},
  pages={795--811},
  year={2016},
  organization={Springer}
}

@article{ryu2016primer,
  title={Primer on monotone operator methods},
  author={Ryu, Ernest K and Boyd, Stephen},
  journal={Appl. comput. math},
  volume={15},
  number={1},
  pages={3--43},
  year={2016}
}

@book{nesterov2018lectures,
  title={Lectures on convex optimization},
  author={Nesterov, Yurii and others},
  volume={137},
  year={2018},
  publisher={Springer}
}

@article{qi2019new,
  title={New models for future problems solving by using ZND method, correction strategy and extrapolation formulas},
  author={Qi, Zhiyuan and Zhang, Yunong},
  journal={IEEE Access},
  volume={7},
  pages={84536--84544},
  year={2019},
  publisher={IEEE}
}

@inproceedings{zhang2019presentation,
  title={Presentation, derivation and numerical experiments of a group of extrapolation formulas},
  author={Zhang, Yunong and Yang, Zhicheng and Chen, Jianrong and Qi, Zhiyuan and Yang, Guanqun},
  booktitle={2019 9th International Conference on Information Science and Technology (ICIST)},
  pages={240--246},
  year={2019},
  organization={IEEE}
}

@inproceedings{hours2014parametric,
  title={A parametric multi-convex splitting technique with application to real-time NMPC},
  author={Hours, Jean-Hubert and Jones, Colin N},
  booktitle={53rd IEEE Conference on Decision and Control},
  pages={5052--5057},
  year={2014},
  organization={IEEE}
}

@article{jakubiec2012d,
  title={D-map: Distributed maximum a posteriori probability estimation of dynamic systems},
  author={Jakubiec, Felicia Y and Ribeiro, Alejandro},
  journal={IEEE Transactions on Signal Processing},
  volume={61},
  number={2},
  pages={450--466},
  year={2012},
  publisher={IEEE}
}

@article{dall2016optimal,
  title={Optimal power flow pursuit},
  author={Dall'Anese, Emiliano and Simonetto, Andrea},
  journal={IEEE Transactions on Smart Grid},
  volume={9},
  number={2},
  pages={942--952},
  year={2016},
  publisher={IEEE}
}

@article{simonetto2020time,
  title={Time-varying convex optimization: Time-structured algorithms and applications},
  author={Simonetto, Andrea and Dall'Anese, Emiliano and Paternain, Santiago and Leus, Geert and Giannakis, Georgios B},
  journal={Proceedings of the IEEE},
  volume={108},
  number={11},
  pages={2032--2048},
  year={2020},
  publisher={IEEE}
}

@article{zavala2010real,
  title={Real-time nonlinear optimization as a generalized equation},
  author={Zavala, Victor M and Anitescu, Mihai},
  journal={SIAM Journal on Control and Optimization},
  volume={48},
  number={8},
  pages={5444--5467},
  year={2010},
  publisher={SIAM}
}

@article{dontchev2013euler,
  title={An {Euler}--{Newton} continuation method for tracking solution trajectories of parametric variational inequalities},
  author={Dontchev, Asen L and Krastanov, MI and Rockafellar, R Tyrrell and Veliov, Vladimir M},
  journal={SIAM Journal on Control and Optimization},
  volume={51},
  number={3},
  pages={1823--1840},
  year={2013},
  publisher={SIAM}
}

@article{massicot2019line,
  title={On-line non-convex constrained optimization},
  author={Massicot, Olivier and Marecek, Jakub},
  journal={arXiv preprint arXiv:1909.07492},
  year={2019}
}

@article{tang2022running,
  title={Running primal-dual gradient method for time-varying nonconvex problems},
  author={Tang, Yujie and Dall'Anese, Emiliano and Bernstein, Andrey and Low, Steven},
  journal={SIAM Journal on Control and Optimization},
  volume={60},
  number={4},
  pages={1970--1990},
  year={2022},
  publisher={SIAM}
}

@article{ding2021time,
  title={Time-variation in online nonconvex optimization enables escaping from spurious local minima},
  author={Ding, Yuhao and Lavaei, Javad and Arcak, Murat},
  journal={IEEE Transactions on Automatic Control},
  volume={68},
  number={1},
  pages={156--171},
  year={2021},
  publisher={IEEE}
}

@article{koppel2017d4l,
  title={D4l: Decentralized dynamic discriminative dictionary learning},
  author={Koppel, Alec and Warnell, Garrett and Stump, Ethan and Ribeiro, Alejandro},
  journal={IEEE Transactions on Signal and Information Processing over Networks},
  volume={3},
  number={4},
  pages={728--743},
  year={2017},
  publisher={IEEE}
}

@article{dixit2019online,
  title={Online learning with inexact proximal online gradient descent algorithms},
  author={Dixit, Rishabh and Bedi, Amrit Singh and Tripathi, Ruchi and Rajawat, Ketan},
  journal={IEEE Transactions on Signal Processing},
  volume={67},
  number={5},
  pages={1338--1352},
  year={2019},
  publisher={IEEE}
}

@book{suli2003introduction,
  title={An introduction to numerical analysis},
  author={S{\"u}li, Endre and Mayers, David F},
  year={2003},
  publisher={Cambridge university press}
}

@article{geman1985bayesian,
  title={Bayesian image analysis: An application to single photon emission tomography},
  author={Geman, Stuart and McClure, Donald E.},
  journal={Proceedings of the American Statistical Association},
  pages={12--18},
  year={1985}
}

@article{jin2016continuous,
  title={Continuous and discrete Zhang dynamics for real-time varying nonlinear optimization},
  author={Jin, Long and Zhang, Yunong},
  journal={Numerical Algorithms},
  volume={73},
  number={1},
  pages={115--140},
  year={2016},
  publisher={Springer}
}

@article{liao2015taylor,
  title={Taylor \ensuremath{O(h^3)} discretization of ZNN models for dynamic equality-constrained quadratic programming with application to manipulators},
  author={Liao, Bolin and Zhang, Yunong and Jin, Long},
  journal={IEEE transactions on neural networks and learning systems},
  volume={27},
  number={2},
  pages={225--237},
  year={2015},
  publisher={IEEE}
}

@article{guo2018design,
  title={Design and analysis of two discrete-time ZD algorithms for time-varying nonlinear minimization},
  author={Guo, Dongsheng and Lin, Xinjie and Su, Zhaozhu and Sun, Sibo and Huang, Zhijing},
  journal={Numerical Algorithms},
  volume={77},
  number={1},
  pages={23--36},
  year={2018},
  publisher={Springer}
}

@article{qiu2019new,
  title={New five-step DTZD algorithm for future nonlinear minimization with quartic steady-state error pattern},
  author={Qiu, Binbin and Zhang, Yunong and Guo, Jinjin and Yang, Zhi and Li, Xiaodong},
  journal={Numerical Algorithms},
  volume={81},
  number={3},
  pages={1043--1065},
  year={2019},
  publisher={Springer}
}

@article{rostami2025time,
  title={Time-varying convex optimization with \ensuremath{O(n)} computational complexity},
  author={Rostami, Mohammadreza and Kia, Solmaz S},
  journal={Automatica},
  volume={178},
  pages={112355},
  year={2025},
  publisher={Elsevier}
}

@article{bastianello2023internal,
  title={Internal model-based online optimization},
  author={Bastianello, Nicola and Carli, Ruggero and Zampieri, Sandro},
  journal={IEEE Transactions on Automatic Control},
  volume={69},
  number={1},
  pages={689--696},
  year={2023},
  publisher={IEEE}
}

@article{bianchin2024internal,
  title={The internal model principle of time-varying optimization},
  author={Bianchin, Gianluca and Van Scoy, Bryan},
  journal={arXiv preprint arXiv:2407.08037},
  year={2024}
}

@article{simonetto2024nonlinear,
  title={Nonlinear optimization filters for stochastic time-varying convex optimization},
  author={Simonetto, Andrea and Massioni, Paolo},
  journal={International Journal of Robust and Nonlinear Control},
  volume={34},
  number={12},
  pages={8065--8089},
  year={2024},
  publisher={Wiley Online Library}
}

@article{maity2022predictor,
  title={Predictor-corrector algorithms for stochastic optimization under gradual distribution shift},
  author={Maity, Subha and Mukherjee, Debarghya and Banerjee, Moulinath and Sun, Yuekai},
  journal={arXiv preprint arXiv:2205.13575},
  year={2022}
}

@article{hazan2016introduction,
  title={Introduction to online convex optimization},
  author={Hazan, Elad},
  journal={Foundations and Trends in Optimization},
  volume={2},
  number={3-4},
  pages={157--325},
  year={2016},
  publisher={Emerald Publishing Limited}
}
\end{document}